\RequirePackage{fix-cm}
\documentclass[onecolumn]{svjour3}          
\smartqed  
\usepackage{graphicx}
\usepackage{amssymb}

\usepackage{amssymb,amsmath,graphicx,epsfig,psfrag}

\def\ds{\displaystyle}

\newcommand{\bw}{{{\bf{  w}}}}

\begin{document}

\title{ALE-SUPG finite element method for convection-diffusion problems in time-dependent domains: Non-conservative form}


\author{ Sashikumaar~Ganesan \and Shweta Srivastava$^*$ }


\institute{S. Ganesan \and S. Srivastava$^*$   \at 
 \email{sashi@serc.iisc.in}           
                    \at
           \email{shweta@nmsc.serc.iisc.in}\\
              Numerical Mathematics and Scientific Computing, \\
Department of Computational and Data Sciences, Indian Institute of Science\\
 Bangalore  560012, India.\\
              Tel.: +91-80-22932902 \\
              Fax: +91-80-23602648 \\             
}

\date{Received: date / Accepted: date}

\maketitle
\begin{abstract}
Stability estimates for Streamline Upwind Petrov-Galerkin (SUPG) finite element method with different time integration schemes for the solution of a scalar transient convection-diffusion-reaction equation in a time-dependent domain are derived. The deformation of the domain is handled with the arbitrary Lagrangian-Eulerian (ALE) approach. In particular, the non-conservative form of the ALE scheme is considered. The implicit Euler, the Crank-Nicolson, and the backward-difference~(BDF-2) methods are used for temporal discretization. It is shown that the stability of the semi-discrete (continuous in time) ALE-SUPG equation is independent of the mesh velocity, whereas the stability of the fully discrete problem is only conditionally stable. The theoretical considerations are illustrated by a numerical example. Further, the dependence of numerical solution on the choice of stabilization parameter $\delta_k$ is also presented. 
\end{abstract}


\section{Introduction} 
Solution of a transient convection-diffusion-reaction equation in a  time-dependent domain is highly demanded in many applications. The scalar variable can be the temperature, concentration or density etc. of a species.  Numerical approximations of the scalar partial differential equations become more challenging when the equation is convection dominated. Further, the computations become more complex, when the domain contains moving boundaries e.g., fluid-structure interactions (FSI) problems. In addition to a stabilized numerical method, an efficient approach is necessary to handle the mesh movement for a convection dominated convection-diffusion equation in time-dependent domains.
 
In general, the standard Galerkin finite element approximation induces spurious oscillations in the numerical solution of a convection dominated convection-diffusion equation. Therefore, stabilization schemes such as streamline upwind Petrov-Galerkin (SUPG)\cite{BH82,BUR10,JOH10}, the local projection stabilization (LPS)~\cite{BRA06,GAN10,MAT07}, the continuous interior penalty method (CIP)~\cite{BUR07,BUR04}, the subgrid scale modeling (SGS)~\cite{Guermond99}, and the orthogonal subscales method (OSS)~\cite{CODINA2000,codina98} have been proposed and analyzed  in the literature. Each of these methods has its own advantages and disadvantages. Nevertheless, almost all stabilization methods introduce an unknown numerical parameter that controls the stability. The choice of the parameter for SUPG stabilization scheme has been studied in~\cite{JOH10,TEZ86,TEZ2000,FRANCA92}.

The above mentioned stabilization methods are mostly studied only for PDEs in stationary domains. In this paper,  the SUPG finite element scheme for computations of  transient convection-diffusion equation in time-dependent domains  is presented. Further, the arbitrary Lagrangian-Eulerian (ALE) approach is used to handle the moving boundaries and the time-dependent domain.  
The ALE formulation introduces a convective domain/mesh velocity term into the  model equation, and hence it alters the overall convective field of the problem. Nevertheless, the model problem can still be convection dominated and can have boundary/interior layers even after reformulating into an ALE form. The stability estimates for the conservative ALE-SUPG approach with implicit Euler and Crank-Nicolson time discretizations have been presented in our earlier work~\cite{shw015}. The considered non-conservative ALE-SUPG formulation avoids the necessity of the  Reynolds identity. 
The stability estimates for the implicit Euler, Crank-Nicolson and backward difference time discretizations with inconsistent SUPG for non-conservative ALE form of the convection-diffusion-reaction equation in time-dependent domains are derived.

The paper is organized as follows. In section 2, the transient convection-diffusion equation in a  time-dependent domain and its  ALE formulation are given. The spatial discretization using the SUPG finite element method is also presented in this section. Further, the stability of the semi-discrete (in space) problem is derived. Section 3 is devoted to the stability estimates of the fully discrete problem with implicit Euler, Crank-Nicolson and backward difference (BDF-2) discretization in time.
The numerical results are presented in Section 4. Finally, a brief summary is presented in Section 5. 

\section{Problem statement}
 Let $\Omega_t$ be a time-dependent bounded domain in $R^{d},~d={2,3}$ with Lipschitz boundary $\partial\Omega_t$ for each $t \in [0,\rm{T}]$. Here, $T$ is a given end time. Consider a transient convection-diffusion-reaction equation : find $u(t,x) :  (0,\rm{T}] \times \Omega \rightarrow \mathbb{R}$  
\begin{equation}\label{model}
\begin{array}{rcll}
\ds\frac{\partial u}{\partial t} - \epsilon\Delta u +  \mathbf b \cdot\nabla u + cu &=& f  &\qquad \text{ in} \,\ (0,\rm{T}] \times \Omega_t,\\
                 u &=& 0 &\qquad   \text{ on} \,\ [0,\rm{T}] \times \partial\Omega_t,\\
                 u(0,x) &=& u_{0}(x) & \qquad  \text{  in} \,\ \Omega_0, 
\end{array}
\end{equation}
 Here, $u(t,x)$ is an unknown scalar function, $u_{0}(x)$ is a  given initial data, $\epsilon$ is the diffusivity constant, $ \mathbf b(t,x)$ is the convective flow velocity, $c(t,x)$  is a reaction function, $f(x)$ is a given source term with $ ~f \in L^2(\Omega_t)$. 

\subsection{ALE formulation}
Let $\hat\Omega$  be a reference domain. The reference domain  $\hat\Omega$ can simply be the initial domain $\Omega_0$ or the previous time-step domain, when the deformation of the domain is large. Let $\mathcal{A}_t$ be a family of bijective ALE mappings, which at each time $t \in (0,\rm{T}]$, maps a point $Y$ of a reference domain $\hat\Omega$ to a point on the current domain $\Omega_t$, given by
\[
\mathcal{A}_t:\hat{\Omega} \rightarrow \Omega_t,  \qquad  \mathcal{A}_t(Y)= x(Y, t), \qquad t\in(0,\rm{T}).
\]
Further, for any time $t_1,t_2 \in [0,\rm{T}]$, the ALE mapping between two time levels will be given by,
\[
 \mathcal{A} : \Omega_{t_{1}} \rightarrow \Omega_{t_{2}} \qquad \qquad \mathcal{A}_{t_1, t_2} = \mathcal{A}_{t_2} \circ \mathcal{A}_{t_1}^{-1}
\]
The domain velocity $\bw$ is defined as 
\[
 \bw (x,t) = \frac{\partial x}{\partial t}\Big|_{Y}(\mathcal{A}_t^{-1}(x),t).
\]
We assume that  $\Omega_t$ is bounded with Lipschitz continuous boundary for each  $t\in[0,\rm{T}]$.  For a  function $g\in C^0({{\Omega_t}})$ on the Eulerian frame, we define the corresponding function $\hat g\in C^0({{\hat\Omega} })$ on the ALE frame as 
\[
\hat g :\hat\Omega\times (0, {\rm{T}}) \rightarrow \mathbb{R}, \qquad 
 \hat{g}:=g\circ \mathcal{A}_t,\qquad \text{with} \qquad \hat{g}(Y, t) = g(\mathcal{A}_t(Y), t).
\]
The temporal derivative on the ALE frame is defined as 
\[
\ds\frac{\partial g}{\partial t} \Big|_{Y}:\Omega_t\times (0, {\rm{T}}) \rightarrow \mathbb{R}, \qquad 
\ds\frac{\partial g  }{\partial t} \Big|_{Y }(x,t) = \ds\frac{\partial \hat g }{\partial t}(Y,t), \qquad 
 Y= \mathcal{A}_t^{-1}(x).
 \]
Applying the chain rule to the time derivative of $g\circ\mathcal{A}_t$ on the ALE frame to get
\[
\ds\frac{\partial g}{\partial t} \Big|_{Y} = \ds\frac{\partial g}{\partial t} (x,t)
 + \ds\frac{\partial x}{\partial t}\Big|_{Y}\cdot\nabla_x g  = \ds\frac{\partial g}{\partial t}  
 + \ds\frac{\partial \mathcal{A}_t(Y)}{\partial t}  \cdot\nabla_x g= \ds\frac{\partial g}{\partial t}  
 + \bw\cdot\nabla_x g,
\]
By using the relation in the model problem~\eqref{model}, we get
\begin{equation}\label{ALEmodel}
\ds\frac{\partial u}{\partial t} \Big|_{Y} - \epsilon\Delta u +  (\mathbf b -\bw)\cdot\nabla u + cu = f.
\end{equation}
This equation is the ALE counterpart of the model equation~\eqref{model}. The difference between the equations~\eqref{model} and~\eqref{ALEmodel} is the additional domain velocity in the ALE form that accounts for the deformation of the domain.

\subsection{Variational form}
In this section, the finite element variational form of the ALE equation \eqref{ALEmodel} is derived. Let 
\[
 V = \left\{v\in  H_{0}^{1}(\Omega_{t}),  ~~v:\Omega_t\times (0,\rm{T}] \rightarrow \mathbb{R}, ~~ v=\hat{v}  \circ A_{t}^{-1}, ~~ \hat{v} \in H_{0}^{1}(\hat{\Omega}) \right\}.
\]
be the solution space for equation~\eqref{ALEmodel}.
Multiplying equation~\eqref{ALEmodel} with a test function $v\in V$ and applying integration by parts to the higher order derivative term, the variational form of the equation~\eqref{ALEmodel} becomes:\\

\noindent For given $\mathbf b$, $\bw$, $c$, $u_0$, $\Omega_0$ and $f$, find $u\in V$ such that for all $ t\in(0,\rm{T}]$
\begin{equation}\label{weakALE}
\left(\frac{\partial u}{\partial t} \Big|_{Y},~v\right) + (\epsilon\nabla u, ~\nabla v) + \left(  (\mathbf b -\bw) \cdot \nabla u ,~v\right) + \left(cu, ~v \right) = (f,~v),  \qquad  v\in V.
\end{equation}
Here,  $(\cdot,~\cdot)$ denotes the $L^2-$inner product in $\Omega_t$. 

\subsection{SUPG finite element space discretization} The stability estimates for the standard Galerkin solution of~\eqref{weakALE} can be found in~\cite{BO04,NOB01,MAC12}. It has been shown that the stability inequality is independent of domain velocity. However, the Galerkin approximation suffers instabilities for convection dominant scalar equations of type~\eqref{weakALE}. In order to overcome this, we use the SUPG discretization for the considered ALE equation~\eqref{weakALE}. Note that the convective velocity in the ALE form is $(\mathbf b -\bw)$, whereas the convective velocity in fixed domain will be $ \textquoteleft\textquoteleft {\mathbf b} \textquotedblright$.

Let $\mathcal{T}_{h,0}$ be the triangulation of initial domain $\Omega_0$. For each $t \in (0,\rm{T}]$, $\mathcal{T}_{h,t}$ denote the family of shape regular triangulations of $\Omega_t$ into simplices obtained by triangulating  the time-dependent domain $\Omega_t$. We denote the diameter of the cell $K  \in\mathcal{T}_{h,t}$ by $h_{K,t}$ and the global mesh size in the triangulated  domain $\Omega_{h,t}$   by $h_t:=\max\{h_{K,t}~:~ K \in\mathcal{T}_{h,t}\}$.   Suppose $V_{h}\subset V$  is a  conforming finite element (finite dimensional) space. Let $\phi_h:=\phi_i(x) $, $i=1,2,...,\mathcal{N},$ be the finite element basis functions of $V_{h}$. The discrete finite element space $V_{h}$ is then defined by  
\[
  V_{h}  = 
\left\{v_h \in C(\overline{\Omega_t}):~v_h|_{\partial \Omega_t}= 0;~ v_h |_K \in {P_k}(K)\right\}\subset
H_0^{1}(\Omega_t).
\]
where ${ P_k}$ is the set of polynomials of degree less than or equal to $k$ for discretization of the ALE mapping in space. We next define the discrete ALE mapping $\mathcal{A}_{h,t}(Y)$ and the mesh velocity $\bw_h$ in space. We use the Lagrangian finite element space
\[
 \mathcal{L}^{k}(\hat\Omega) = \left\{  \psi\in H^k(\hat \Omega):  \psi|_K\in{ P_k}(\hat K)  \text{ for all } \hat K\in  \hat \Omega_h \right\},
\]
 Using the linear space, we define the semidiscrete ALE mapping in space for each $t\in[0,{\rm{T})}$ by
\begin{equation} \label{discALE}
 \mathcal{A}_{h,t}:\hat{\Omega}_{h} \rightarrow \Omega_{h,t}.
\end{equation}
Further, the semidiscrete (continuous in time) mesh velocity $\bw_h(t,Y)\in \mathcal{L}^{1}(\hat\Omega)^d$ in the   ALE frame for each $t\in[0,{\rm{T})}$ is defined   by
\[
 \hat\bw_h(t,Y)=   \sum_{i=1}^{\mathcal{M}}\bw_{i}(t)
\psi_i(Y) ;\quad \bw_{i}(t)\in\mathbb{R}^d. 
\]
Here, $\bw_{i}(t)$ denotes the mesh velocity of the $i^{th}$ node of simplices at time $t$, and $\psi_i(Y) $, $i=1,2,...,\mathcal{M},$ are the  basis functions of $\mathcal{L}^{1}(\hat\Omega_h)$. We then define the semidiscrete mesh velocity in the Eulerian frame as 
\[
  \bw_h(t,x) =  \hat\bw_h  \circ\mathcal{A}^{-1}_{h,  t}(x).
\]
Using the above finite element spaces and applying the inconsistent SUPG finite element discretization to the ALE form~\eqref{ALEmodel}, the  semi-discrete form in space reads:\\

\noindent For a given $u_h(x,0)=u_{h,0}$,  $\mathbf b$, $\bw_h$, $c$, $\Omega_{h,0}$  and $f$, find $ u_h(t,x)\in   V_{h}$ such that for all $ t\in(0,\rm{T}]$
\begin{equation} \label{semidisc}
\begin{aligned} 
\left (\left. \frac{\partial u_h}{\partial t} \right \arrowvert_{Y},  v_h\right)    & + a_{SUPG}(u_h,v_h)    -  \int_{\Omega_{h,t}} \bw_h \cdot \nabla u_h~ v_h~  dx \\ 
   & = \int_{\Omega_{h,t}} fv_h~   dx 
+ \sum_{K  \in \mathcal{T}_{h,t}} \delta_{K}  \int_{K} f ~(\mathbf{b-w}_h) \cdot\nabla v_h ~ dK
\end{aligned}
\end{equation}
where
\begin{align}\label{supg}
 a_{SUPG}(u,v)  &= \epsilon(\nabla u, \nabla v) +  (\mathbf{b} \cdot \nabla u, v) + (c u, v) \nonumber\\
 &+\sum_{K \in \mathcal{T}_{h,t}} \delta_{K} (- \epsilon\Delta u + (\mathbf{b-w}_h) \cdot\nabla u + cu, (\mathbf{b-w}_h ) \cdot \nabla v)_K
\end{align}
Here, $\delta_{K}$ is the local stabilization parameter, whose value depends on the
mesh size and the convective velocity. Further, $u_{h,0}\in V_{h}$ is defined as the $L^2$-projection of the initial value $u_0$ onto $V_{h}$. 
 
\begin{lemma} \label{lemma1}
Coercivity of $a_{SUPG}(\cdot , \cdot)$: Assume that there exists a constant $\mu$ such that
\begin{equation}\label{assump1}
  \left(c - \frac{1}{2}\nabla \cdot \mathbf{b}\right)(x) \geq \mu > 0, \quad \forall~x \in \Omega_t.\\
\end{equation}
 Let the discrete form of the assumptions~\eqref{assump1} be satisfied. Further, assume that the SUPG parameter satisfies
\begin{equation}\label{assump2}
 \delta_{K} \le  \frac{\mu_{0}}{2||c||_{K,\infty}^{2}}, \qquad    \delta_{K} \le  \frac{h_{K}^{2}}{2\epsilon c_{inv}^{2}},
\end{equation}
where $c_{inv}$ is a constant used in the inverse inequality.  Then, the SUPG bilinear form  satisfies
\[
 a_{SUPG}(u_h,u_h) \geq \frac{1}{2}|||u_h|||^{2},
\]
where the mesh dependent norm is defined as
\[
 |||u|||^{2} = \left(\epsilon| u|^{2}_{1}  + \sum_{K \in \mathcal{T}_{h,t}} \delta_{K} ||(\mathbf{b-w}_h) \cdot \nabla u||^{2}_{0,K} + \mu||u||_{0}^{2}\right).
\]
\end{lemma}
\begin{proof} Coercivity of the bilinear form has already been proved in~\cite{shw015}.
\end{proof}

\begin{lemma} 
Stability of the semi-discrete (continuous in time) problem: Let the discrete version of~\eqref{assump1} and the assumption~\eqref{assump2} on $\delta_K$ hold true. Then, the solution of the problem~\eqref{semidisc} satisfies,
\[
 ||u_h||^2_{0}  +  \frac{1}{2} \int_0^T  |||u_h|||^2 dt  \leq ||u_h(0)||^2_{0} +\frac{2}{\mu} \int_0^T
 ||f||_{0}^{2} ~dt + 2\int_0^T \sum_{K \in \mathcal{T}_{h,t}}\delta_{K} ||f||_{0}^{2}~ dt  
\]
\end{lemma}

\begin{proof}
Using the relations
\begin{equation*}
 \int_{\Omega_{h,t}} \left.\frac{\partial u_h}{\partial t} \right \arrowvert_{Y}u_h~ d x = \frac{1}{2}\left( \frac{d}{dt}||u_h||_{0}^2 - \int_{\Omega_{h,t}} u_h^2 \nabla \cdot \mathbf{w_h} d x\right)
\end{equation*}
and
\begin{equation*}
  \int_{\Omega_{h,t}}\mathbf{w_h}\cdot \nabla u_h~u_h ~dx = -\frac{1}{2}\int_{\Omega_{h,t}} u_h^2 ~\nabla \cdot \mathbf{w}_h ~d x
\end{equation*}
for equation~\eqref{semidisc} and following the similar procedure, as in the case of conservative ALE formulation~\cite{shw015} in section $(3)$, the stability estimate for the semi-discrete problem can be derived. Hence the stability properties are not affected by domain velocity field in the semi-discrete problem\label{semidisc}. However, we may not expect the same result to be true for the fully discrete case.
\end{proof}
\section{Fully discrete scheme} In this section, the stability estimates for the fully discrete ALE-SUPG form is derived. First, we consider the first order implicit Euler for the temporal discretization and then the second order modified Crank-Nicolson, and backward-difference (BDF-2) method.

Consider the partition of time interval $[0,\rm{T}]$ as $0=t^0<t^1<\dots <t^N=\rm{T}$
into $N$ equal time intervals. Let us denote the uniform time step by $\Delta t = \tau^n   $ = $t^{n}$ - $t^{n-1}$, $1\le n \le N$. Further, let   $u_h^n$ be the
approximation of $u(t^n,x)$ in $V_{h}\subset
H_0^{1}(\Omega_{t^n})$, where $\Omega_{t^n}$ is the deforming domain at time $t=t^n$. We first discretize  the ALE mapping in time using a linear interpolation. We denote the discrete ALE mapping by 
$\mathcal{A}_{h, \Delta t}$, and define it for every $\tau\in[t^n,t^{n+1}]$ by
\[
 \mathcal{A}_{h, \Delta t}(Y) = \frac{\tau-t^n}{\Delta t}\mathcal{A}_{h,t^{n+1}}(Y) +  \frac{t^{n+1}-\tau }{\Delta t}\mathcal{A}_{h,t^{n}}(Y),
\]
where $\mathcal{A}_{h,t}(Y)$ is the time continuous ALE mapping defined in~\eqref{discALE}. Since the ALE mapping is discretized in time using a linear interpolation, we obtain the discrete mesh velocity 
\begin{equation}
  \hat\bw_h^{n+1}(Y)= \frac{\mathcal{A}_{h,t^{n+1}}(Y) - \mathcal{A}_{h,t^{n}}(Y)}{\Delta t} = \frac{x_h^{n+1}- x_h^{n}}{\Delta t}
\end{equation}
as a piecewise constant function in time. Further, we define the mesh velocity on the Eulerian frame as
\[
  \bw_h^{n+1} =  \hat\bw_h^{n+1}  \circ\mathcal{A}^{-1}_{h,\Delta t}(x).
\]
Further, the integrals $u_h^n$ on a domain $\Omega_{t^s}$ with $t^s \neq t^n$ is written through the ALE mapping
\[
 \int_{\Omega_{t^s}} u_h^n ~dX \colon =  \int_{\Omega_{t^s}} u_h^n \circ \mathcal{A}_{t^n, t^s}~ dX.
\]

\subsection{Discrete ALE-SUPG with Implicit Euler time discretization method}
Applying the backward Euler time discretization to the semi-discrete problem~\eqref{semidisc}, the discrete form  reads: 
\begin{equation} \label{disc}
\begin{aligned} 
&\left (\ds\frac{  u^{n+1}_h - u^n_h}{\Delta t},  v_h\right)_{\Omega_{t ^{n+1}}}    + a^{n+1}_{SUPG}(u^{n+1}_h,v_h)    -   \int_{\Omega_{t^{n+1}}} \bw^{n+1}_h \cdot \nabla u^{n+1}_h~ v_h~  dx \\ 
   &\qquad \qquad = \int_{\Omega_{t^{n+1}}} f^{n+1} v_h~   dx 
+ \sum_{K  \in \mathcal{T}_{t^{n+1}}} \delta_{K}  \int_{K} f^{n+1} ~(\mathbf{b-w}^{n+1}_h) \cdot\nabla v_h ~ dK,
\end{aligned}
\end{equation}
where
\begin{align*} 
 &a^{n+1}_{SUPG}(u_h,v_h)  = \epsilon(\nabla u_h, \nabla v_h)_{\Omega_{t ^{n+1}}} +  (\mathbf{b} \cdot \nabla u_h, v_h)_{\Omega_{t ^{n+1}}} + (c u_h, v_h)_{\Omega_{t^{n+1}}} \nonumber\\
 & \qquad +\sum_{K \in \mathcal{T}_{t^{n+1}}} \delta_{K} (- \epsilon\Delta u_h + (\mathbf{b-w}^{n+1}_h) \cdot\nabla u_h + cu_h, ~(\mathbf{b-w}^{n+1}_h ) \cdot \nabla v_h)_K.
\end{align*}

\begin{lemma}\label{stabnoncons} {Stability estimates for non-conservative ALE-SUPG form with implicit Euler method:}
Let the discrete version of~\eqref{assump1} and the assumption~\eqref{assump2} on $\delta_K$ hold true.  Further, assume that $\delta_K \leq \frac{\Delta t}{4}$ then the   solution of the problem~\eqref{disc} satisfies
\begin{align*} 
   \|u_h^{n+1}\|&^2_{L^2(\Omega_{n+1})}  + \frac{\Delta t}{2} \sum_{i=1}^{n+1}|||u_h^i|||^{2}_{L^2 \left(\Omega_{t^i} \right)}   \\ 
 \quad & \leq \left((1+\Delta t \alpha_2^0)\| u_h^{0}\|^2_{L^2(\Omega_{0})}  + \Delta t\sum_{i=1}^{n+1}\left(\frac{2}{\mu}+  \frac{\Delta t}{2}\right) \|f^{i}\|^2_{L^{2}(\Omega_{i})}\right) \\
& \qquad \qquad\qquad\qquad\qquad\qquad\qquad \exp{\left( \Delta   t \sum_{i=1}^{n+1}\frac{\alpha_1^i + \alpha_2^i}{1-\Delta t(\alpha_1^i + \alpha_2^i)} \right)},
\end{align*}
where $\alpha_1^n $ and $\alpha_2^n$ are defined as in the proof of this lemma.
\end{lemma}
\begin{proof}
 Substituting $v_h = u_h^{n+1}$ in the discrete form~\eqref{disc} and after applying the integration by parts to the mesh velocity integral, we get
\begin{align*}
&\left (\ds\frac{  u^{n+1}_h - u^n_h}{\Delta t},  u_h^{n+1}\right)_{\Omega_{t ^{n+1}}}    + a^{n+1}_{SUPG}(u^{n+1}_h,u_h^{n+1})    +  \frac{1}{2} \int_{\Omega_{t^{n+1}}}  \nabla\cdot\bw^{n+1}_h  |u_h^{n+1}|^2~  dx \\ 
   & \qquad\qquad= \int_{\Omega_{t^{n+1}}} f^{n+1} u_h^{n+1} dx 
+ \sum_{K  \in \mathcal{T}_{t^{n+1}}} \delta_{K}  \int_{K} f^{n+1} ~(\mathbf{b-w}^{n+1}_h) \cdot\nabla u_h^{n+1} dK.
\end{align*}
Using  the coercivity of bilinear form $a_{SUPG}$ and applying Cauchy-Schwarz inequality, we get
\begin{align*}
&||u_h^{n+1}||^2_{L^2 \left( \Omega_{t^{n+1}}\right)} + \frac{\Delta t}{2}  |||u_h^{n+1}|||^2_{L^2 \left( \Omega_{t^{n+1}} \right)}  \\
&~  \leq - \frac{1}{2} \Delta t \int_{\Omega_{t^{n+1}}} \nabla \cdot \mathbf{w}_h^{n+1}|u_h^{n+1}|^2 ~dx +\frac{1}{2}|| u_h^{n}||^2_{L^2 \left( \Omega_{t^{n+1}} \right)} + \frac{1}{2}||u_h^{n+1}||^2_{L^2 \left( \Omega_{t^{n+1}} \right)} \\
&\quad+\frac{\Delta t }{4}\sum_{K \in \mathcal{T}_{h,t^{n+1}}}\delta_{K} ||(\mathbf{b-w}_h^{n+1}) \cdot \nabla u_h^{n+1}||^2 + \Delta t \sum_{K \in \mathcal{T}_{h,t^{n+1}}}  \delta_{K} ||f^{n+1}||^2_{L^2 \left( \Omega_{t^{n+1}} \right)}\\
&\quad +  \frac{\Delta t}{ \mu}||f^{n+1}||^2_{L^{2}(\Omega_{t^{n+1}})} + \Delta t \frac{\mu}{4} || u_h^{n+1} ||^2_{L^2 \left( \Omega_{t^{n+1}} \right)}.
\end{align*}
Since we have
\begin{align*}
 || u_h^{n}||^2_{L^2 \left( \Omega_{t^{n+1}} \right)} = || u_h^{n}||^2_{L^2( \Omega_{t^{n}})} + \int_{t^{n}}^{t^{n+1}}\int_{\Omega_{t}} \nabla \cdot \mathbf{w}_{h}|u_h^{n}|^2~dx~ dt,
\end{align*}
 we get
\begin{align*}
||u_h^{n+1}||&^2_{L^2 \left( \Omega_{t^{n+1}} \right)} + \frac{1}{2} \Delta t |||u_h^{n+1}|||^2_{L^2 \left( \Omega_{t^{n+1}} \right)} \\
& \leq \int_{t^{n}}^{t^{n+1}}\int_{\Omega_{t}}\nabla \cdot \mathbf{w}_{h}|u_h^{n}|^2~dx~ dt -\Delta t \int_{\Omega_{t^{n+1}}} \nabla \cdot \mathbf{w}_h^{n+1}|u_h^{n+1}|^2 ~dx  \\
& \quad+ || u_h^{n}||^2_{L^2( \Omega_{t^{n}})} +  \Delta t \frac{2}{ \mu}||f^{n+1}||^2_{L^{2}(\Omega_{t^{n+1}})}+ 2\Delta t \sum_{K \in \mathcal{T}_{h,t^{n+1}}}  \delta_{K} ||f^{n+1}||^2_{L^2 \left( \Omega_{t^{n+1}} \right)}.
\end{align*}
Let   
\[
 \mathcal A_{{t_{n}},t_{n+1}} =\mathcal A_{h, t_{n+1}} \circ~ \mathcal A_{t_{n}}^{-1}
\]
 be the ALE mapping between   $\Omega_{t^{n}}$ and $\Omega_{t^{n+1}}$,  and $J_{\mathcal A_{{t_{n}},t_{n+1}}}$ be its Jacobian, then we have
\begin{align*}
&||u_h^{n+1}||^2_{L^2 \left( \Omega_{t^{n+1}} \right)}  + \frac{1}{2} \Delta t |||u_h^{n+1}|||^2_{L^2 \left( \Omega_{t^{n+1}} \right)}\\
& \quad \leq \Delta t ||\nabla \cdot \mathbf{w}_{h}(t^{n+1})||_{L_{\infty} (\Omega_{t^{n+1}})} 
||u_h^{n+1}||^2_{L^2 \left( \Omega_{t^{n+1}} \right)}+  \Delta t \frac{2}{ \mu}||f^{n+1}||^2_{L^{2}(\Omega_{t^{n+1}})}\\
& \qquad+\left(1+ \Delta t \sup_{t \in (t^{n},t^{n+1})}~ ||J_{\mathcal A_{{t_{n}},t_{n+1}}}\nabla \cdot \mathbf{w}_{h}||_{L_{\infty} (\Omega_{t})}\right)||u_h^{n}||^2_{L^2( \Omega_{t^{n}})}\\
&\qquad + \Delta t \sum_{K \in \mathcal{T}_{h,t^{n+1}}}  \delta_{K} ||f^{n+1}||^2_{L^2 \left( \Omega_{t^{n+1}} \right)}.
\end{align*}
Further, using the notations 
\[
 \alpha_1^n   = ||\nabla \cdot \mathbf{w}_{h}(t^{n})||_{L_{\infty}(\Omega_{t^{n}})}, \qquad
\alpha_2^n = \sup_{t \in (t^{n},t^{n+1})}~ ||J_{\mathcal A_{{t_{n}},t_{n+1}}}\nabla \cdot \mathbf{w}_{h}||_{L_{\infty}(\Omega_{t})},
\]
the above equation can be written as 
\begin{align*}
||u_h^{n+1}||^2_{L^2 \left( \Omega_{t^{n+1}} \right)}& + \frac{1}{2} \Delta t |||u_h^{n+1}|||^2_{L^2 \left( \Omega_{t^{n+1}} \right)} \\
 &  \leq \Delta t \alpha_1^{n+1}||u_h^{n+1}||^2_{L^2 \left( \Omega_{t^{n+1}} \right)}+ \left(1+ \Delta t \alpha_2^n\right)||u_h^{n}||^2_{L^2( \Omega_{t^{n}})}  \\
 & \qquad +  \Delta t \frac{2}{ \mu}||f^{n+1}||^2_{L^{2}(\Omega_{t^{n+1}})}
 + 2\Delta t \sum_{K \in \mathcal{T}_{h,t^{n+1}}}  \delta_{K} ||f^{n+1}||^2_{L^2 \left( \Omega_{t^{n+1}} \right)}.
\end{align*}
Summing over the index $i = 0,1,2, \cdots ,n$, and using the assumptions on $\delta_K$, we have 
\begin{align*}
||&u_h^{n+1}||^2_{L^2 \left( \Omega_{t^{n+1}} \right)} + \frac{1}{2} \Delta t\sum_{i = 1}^{n+1} |||u_h^{i}|||^2_{L^2 \left( \Omega_{t^{i}} \right)} \\
 & \qquad \qquad \leq \Delta t \alpha_1^{n+1}||u_h^{n+1}||^2_{L^2 \left( \Omega_{t^{n+1}} \right)}+\Delta t \sum_{i=1}^{n}(\alpha_1^{i}+   \alpha_2^i)||u_h^{i}||^2_{L^2( \Omega_{t^{i}})}\\
&\qquad \qquad(1+  \Delta t \alpha_2^0)||u_h^{0}||^2_{L^2( \Omega_{t^{0}})}+ 2\Delta t \sum_{i=1}^{n+1}\sum_{K \in \mathcal{T}_{h,t^{i}}}  \delta_{K} ||f^{n}||^2_{L^2 \left( \Omega_{t^{i}} \right)}\\
& \qquad \qquad +   \Delta t \sum_{i=1}^{n+1}\frac{2}{ \mu}||f^{i}||^2_{L^{2}(\Omega_{t^{i}})}\\
& \qquad \qquad \leq \Delta t \sum_{i = 1}^{n+1}(\alpha_1^{i}+  \alpha_2^i)||u_h^{i}||^2_{L^2( \Omega_{t^{i}})} + (1+ \Delta t \alpha_2^0)||u_h^0||^2_{L^2(\Omega_{t^{0}})}\\
& \qquad \qquad + \sum_{i=1}^{n+1}  \left(\frac{2 \Delta t}{ \mu} + \frac{\Delta t ^2}{2}\right)||f^{i}||^2_{L^{2}(\Omega_{t^{i}})}.
\end{align*}
We now apply the Gronwall's lemma to get
\begin{align*} 
  \|u_h^{n+1}\|^2_{L^2(\Omega_{n+1})} & + \frac{\Delta t}{2} \sum_{i=1}^{n+1}|||u_h^i|||^{2}_{L^2 \left(\Omega_{t^i} \right)}   \\ 
 & \leq \left[(1+\Delta t \alpha_2^0)\| u_h^{0}\|^2_{L^2(\Omega_{0})}  + \Delta t\sum_{i=1}^{n+1}\left(\frac{2}{\mu}+  \frac{\Delta t}{2}\right) \|f^{i}\|^2_{L^{2}(\Omega_{i})}\right]\\ 
&   \exp{\left( \Delta   t \sum_{i=1}^{n+1}\frac{\alpha_1^i + \alpha_2^i}{1-\Delta t(\alpha_1^i + \alpha_2^i)} \right)}.
\end{align*}
 The above stability estimate is stable provided
\[
 \Delta t < \frac{1}{\alpha^n_1 + \alpha_2^n} = \left(\|\nabla \cdot \bw_h(t^n) \|_{L^\infty(\Omega_{(t^n)})} + \sup_{t\in(t^{n},t^{n+1})  } \| J_{A_{t^{n},t^{n+1}} }\nabla \cdot  \bw_h \|_{L^\infty(\Omega_t)}\right)^{-1}.
\]
\end{proof}

\begin{remark}
The stability estimates for the consistent ALE SUPG, with the transient term in stabilization can also be proved. Since both the transient terms will be on $\Omega_{n+1}$ domain, it can be handled easily. It can be shown that the stability estimate differs only by a constant multiplication of right hand side terms. Here, the proof is not given for the sake of brevity.
\end{remark}

\subsection{Discrete ALE-SUPG with Crank-Nikolson time discretization} We now consider the modified Crank-Nikolson method which is basically Runge-Kutta method of order two. For an equation
\begin{align}\label{standardeqn}
 \frac{du(t)}{dt} = f(u(t),t),~~~ t>0~~ and ~~u(0) = u_0
\end{align}
 with the modified Crank-Nikolson method, we get
\begin{align*}
 u^{n+1} - u^{n} = \Delta t f \left(\frac{u^{n+1} + u^{n}}{2}, t^{n+\frac{1}{2}} \right)
\end{align*}

\begin{lemma}Stability estimates for the non-conservative ALE-SUPG form applying Crank-Nicolson method:\label{stabnoncons_CN}
 Let the discrete version of~\eqref{assump1} and the assumption~\eqref{assump2} on $\delta_K$ hold true.  Further, assume that $\delta_K \leq \frac{\Delta t}{4}$ then the solution obtained from the Crank-Nicolson time discretization satisfies
\begin{equation*}
\begin{aligned}
  \|u_h^{n+1}\|^2_{L^2(\Omega_{n+1})}  +& \frac{\Delta t}{4} \sum_{i=0}^{n}|||(u_h^{i+1} + u_h^i)|||^{2}_{L^2 \left(\Omega_{t^{i+1/2}} \right)} \\
&\leq  \left((1+\Delta t \beta_2^0)\| u_h^{0}\|^2_{L^2(\Omega_{0})}  + \Delta t\sum_{i=0}^{n}\left(\frac{2}{\mu}+  \Delta t \right) \|f^{i+1/2}\|^2_{L^{2}(\Omega_{i+1/2})}\right)\\
& \exp{\left( \Delta   t \sum_{i=0}^{n}\frac{\beta_1^i + \beta_2^i}{1-\Delta t(\beta_1^i + \beta_2^i)} \right)}. \\
 \end{aligned}
\end{equation*}
\end{lemma}
\begin{proof}
Applying the modified Crank-Nicolson time discretization to the semi-discrete equation $(5)$, we get, 
\begin{equation*}
\begin{aligned}
 &\left (\ds\frac{  u^{n+1}_h - u^n_h}{\Delta t},  v_h\right)_{\Omega_{t^{n+1}}} + a^{n+1/2}_{SUPG}\left(\frac{u^{n+1}_h + u^{n}_h}{2},v_h \right) \\
&\quad- \int_{\Omega_{t^{n+1/2}}} \bw^{n+1/2}_h \cdot \nabla \left(\frac{u^{n+1}_h + u^{n}_h}{2} \right)~ v_h~  dx     \\ 
&\quad  = \int_{\Omega_{t^{n+1/2}}} f^{n+1/2} v_h~   dx + \sum_{K  \in \mathcal{T}_{h,t^{n+1/2}}} \delta_{K}  \int_{K} f^{n+1/2} ~(\mathbf{b-w}_h) \cdot\nabla v_h ~ dK.
\end{aligned}
\end{equation*}
 Testing the above equation with $v_h = (u_h^{n+1} + u_h^n)$, and using the equalities
\begin{equation}
(u_h, u_h + v_h) = \frac{1}{2}||u_h||^2 + \frac{1}{2}||u_h + v_h||^2 - \frac{1}{2}||v_h||^2.  
\end{equation}
Using the relation,
\begin{align*}
 || u_h^{n}||^2_{L^2 \left( \Omega_{t^{n+1}} \right)} = || u_h^{n}||^2_{L^2( \Omega_{t^{n}})} + \int_{t^{n}}^{t^{n+1}}\int_{\Omega_{t}} \nabla \cdot \mathbf{w}_{h}|u_h^{n}|^2~dx~ dt,
\end{align*}
 the first term can be written as,
 \begin{equation*}
\begin{aligned}
 \int_{\Omega_{t^{n+1}}}u_h^{n+1} &(u_h^{n+1} + u_h^n)~dx  - \int_{\Omega_{t^{n+1}}}u_h^{n}(u_h^{n+1} + u_h^n)~dx \\
&= \frac{1}{2} ||u_h^{n+1}||^2_{L_2(\Omega_{t^{n+1}})} + \frac{1}{2}||u_h^{n+1} + u_h^{n}||^2_{L_2(\Omega_{t^{n+1}})}\\& - \frac{1}{2}||u_h^n||^2_{L_2(\Omega_{t^{n+1}})} - \frac{1}{2}||u_h^n||^2_{L_2(\Omega_{t^{n+1}})} \\
&- \frac{1}{2}||u_h^{n+1} + u_h^{n}||^2_{L_2(\Omega_{t^{n+1}})} +  \frac{1}{2}||u_h^{n+1}||^2_{L_2(\Omega_{t^{n+1}})}\\
& = ||u_h^{n+1}||^2_{L_2(\Omega_{t^{n+1}})} - ||u_h^n||^2_{L_2(\Omega_{t^{n+1}})}\\
& = ||u_h^{n+1}||^2_{L_2(\Omega_{t^{n+1}})} - ||u_h^n||^2_{L_2(\Omega_{t^{n}})} - \Delta t\int_{\Omega_{t^{n+1/2}}} \nabla \cdot \mathbf{w}_{h}|u_h^{n}|^2~dx.
\end{aligned}
\end{equation*}
Using this relation, the coercivity of the bilinear form and the Cauchy Schwarz inequality for the right hand side terms, the above equation takes the form,
\begin{equation*}
\begin{aligned}
||u_h^{n+1}||&^2_{L_2(\Omega_{t^{n+1}})} - ||u_h^n||^2_{L_2(\Omega_{t^{n}})} - \Delta t\int_{\Omega_{t^{n+1/2}}} \nabla \cdot \mathbf{w}_{h}|u_h^{n}|^2~dx + \frac{\Delta t}{4} |||(u_h^{n+1} + u_h^n)|||^2_{{L_2(\Omega_{t^{n+1/2}})}} \\
&\leq \frac{\Delta t}{2} \int_{\Omega_{t^{n+1/2}}} \left(\bw_h \cdot \nabla (u^{n+1}_h + u^{n}_h )\right)~ (u^{n+1}_h + u^{n}_h )~  dx 
 + \frac{\Delta t}{\mu} ||f^{n+1/2}||^2_{{L_2(\Omega_{t^{n+1/2}})}} \\&+ \frac{\mu \Delta t}{8}||u_h^{n+1}+u_h^{n}||^2_{{L_2(\Omega_{t^{n+1/2}})}} +\Delta t \sum_{K \in \mathcal{T}_{t^{n+1/2}}}\delta_{K} ||f^{n+1/2}||^2_{K}\\
&
+ \frac{\Delta t }{8} \sum_{K \in \mathcal{T}_{t^{n+1/2}}}\delta_{K} ||(\mathbf{b-w}_h) \cdot \nabla (u_h^{n+1}+ u_h^n)||^2_{K}.
\end{aligned}
\end{equation*}
Absorbing the right hand side terms into the SUPG norm, and using integration by parts for the mesh velocity term, we get,

\begin{equation*}
\begin{aligned}
||u_h^{n+1}&||^2_{L_2(\Omega_{t^{n+1}})} + \frac{\Delta t}{8} |||(u_h^{n+1} + u_h^n)|||^2_{{L_2(\Omega_{t^{n+1/2}})}}\\  
&\leq  \Delta t\int_{\Omega_{h,t^{n+1/2}}} \nabla \cdot \mathbf{w}_{h}|u_h^{n}|^2~dx - \frac{\Delta t}{4}\int_{\Omega_{h, t^{n+1/2}}}\nabla \cdot \mathbf{w}_{h} |u_h^{n+1}+u_h^{n}|^2~dx\\
&   + ||u_h^n||^2_{L_2(\Omega_{t^{n}})} +\frac{\Delta t}{\mu} ||f^{n+1/2}||^2_{{L_2(\Omega_{t^{n+1/2}})}} +\Delta t \sum_{K \in \mathcal{T}_{h,t^{n+1/2}}} \delta_{K} ||f^{n+1/2}||^2_{K}\\
&\leq  \Delta t\int_{\Omega_{h,t^{n+1/2}}} \nabla \cdot \mathbf{w}_{h}\left(|u_h^{n}|^2- \frac{1}{4}|u_h^{n+1}+u_h^{n}|^2 \right)~dx \\
&  +  ||u_h^n||^2_{L_2(\Omega_{t^{n}})} +\frac{\Delta t}{\mu} ||f^{n+1/2}||^2_{{L_2(\Omega_{t^{n+1/2}})}} +\Delta t \sum_{K \in \mathcal{T}_{h,t^{n+1/2}}} \delta_{K} ||f^{n+1/2}||^2_{K}\\
&\leq  \Delta t\int_{\Omega_{h,t^{n+1/2}}} \nabla \cdot \mathbf{w}_{h}\left(|u_h^{n}|^2 + |u_h^{n+1}|^2 \right)~dx +  ||u_h^n||^2_{L_2(\Omega_{t^{n}})}\\
&   +\frac{\Delta t}{\mu} ||f^{n+1/2}||^2_{{L_2(\Omega_{t^{n+1/2}})}} +\Delta t \sum_{K \in \mathcal{T}_{h,t^{n+1/2}}} \delta_{K} ||f^{n+1/2}||^2_{K}.
\end{aligned}
\end{equation*}
Using the ALE map and its Jacobian, the equation becomes
\begin{equation*}
\begin{aligned}
||u_h^{n+1}||^2_{L_2(\Omega_{t^{n+1}})} +& \frac{\Delta t}{4} |||(u_h^{n+1} + u_h^n)|||^2_{{L_2(\Omega_{t^{n+1/2}})}}  \\
 &\leq   \Delta t ||\nabla \cdot \mathbf{w}_{h}||_{L_{\infty}(\Omega_{t^{n+1/2}})}||u_h^{n+1}||^2_{L_2(\Omega_{t^{n+1/2}})}\\
& +\Delta t ||\nabla \cdot \mathbf{w}_{h}||_{L_{\infty}(\Omega_{t^{n+1/2}})} ||u_h^n||^2_{L_2(\Omega_{t^{n+1/2}})} +\frac{\Delta t}{\mu} ||f^{n+1/2}||^2_{{L_2(\Omega_{t^{n+1/2}})}}\\
 &  +\Delta t \sum_{K \in \mathcal{T}_{t^{n+1/2}}} \delta_{K} ||f^{n+1/2}||^2_{K} .
\end{aligned}
\end{equation*}
Further, with the notations
\begin{align*}
 &\beta_1^{n+1}  =  \Big|\Big|J_{\mathcal A_{{t_{n+1}},~t_{n+1/2}}}\Big|\Big|_{L_{\infty}(\Omega_{t^{n+1}})}||\nabla \cdot \mathbf{w}_{h}||_{L_{\infty}(\Omega_{t^{n+1/2}})},\\ 
&\beta_2^n =  \Big|\Big| J_{\mathcal A_{{t_{n}},~t_{n+1/2}}} \Big| \Big|_{L_{\infty}(\Omega_{t^{n}})} 
  || \nabla \cdot \mathbf{w}_{h}||_{L_{\infty}(\Omega_{t^n})},
\end{align*}

the inequality becomes
\begin{equation*}
\begin{aligned}
||u_h^{n+1}&||^2_{L_2(\Omega_{t^{n+1}})} + \frac{\Delta t}{4} |||(u_h^{n+1} + u_h^n)|||^2_{{L_2(\Omega_{t^{n+1/2}})}} \\ 
&\leq  \Delta t \beta_1^{n+1}||u_h^{n+1}||^2_{{L_2(\Omega_{t^{n+1}})}} + ( 1 + \Delta t \beta_2^n)||u_h^n||^2_{L_2(\Omega_{t^{n}})}\\
&  +\frac{\Delta t}{\mu} ||f^{n+1/2}||^2_{{L_2(\Omega_{t^{n+1/2}})}} +2\Delta t \sum_{K \in \mathcal{T}_{t^{n+1/2}}} \delta_{K} ||f^{n+1/2}||^2_{K}.
\end{aligned}
\end{equation*}
Summing over the index $i = 0, 1,2,...n$, and using the assumption on $\delta_k$ , we have
\begin{equation*}
\begin{aligned}
||u_h^{n+1}&||^2_{L_2(\Omega_{t^{n+1}})} + \frac{\Delta t}{4} \sum_{i=0}^{n} |||(u_h^{i+1} + u_h^i)|||^2_{{L_2(\Omega_{t^{i+1/2}})}}  \\
&\leq  \Delta t \beta_1^{n+1}||u_h^{n+1}||^2_{{L_2(\Omega_{t^{n+1}})}} + \Delta t \sum_{i=1}^{n} (\beta_1^i + \beta_2^i)||u_h^i||^2_{{L_2(\Omega_{t^{i}})}} \\
 &  + ( 1 + \Delta t \beta_2^0)||u_h^0||^2_{L_2(\Omega_{t^{0}})}+\sum_{i=0}^{n}\left(\frac{\Delta t}{\mu} ||f^{i+1/2}||^2_{{L_2(\Omega_{t^{i+1/2}})}} +\frac{\Delta t^2}{2} ||f^{i+1/2}||^2_{K}\right)\\
 & \leq \Delta t \sum_{i=1}^{n+1} (\beta_1^i + \beta_2^i)||u_h^i||^2_{{L_2(\Omega_{t^{i}})}} + ( 1 + \Delta t \beta_2^0)||u_h^0||^2_{L_2(\Omega_{t^{0}})}\\
&~ +\Delta t\sum_{i=0}^{n}\left(\frac{2}{\mu} + \frac{\Delta t}{2} \right) \|f^{i+1/2}\|^2_{L^{2}(\Omega_{i+1/2})}.
\end{aligned}
\end{equation*}
Finally, using the Grownwall's lemma, we get
\begin{equation*}
\begin{aligned}
  \|u_h^{n+1}\|&^2_{L^2(\Omega_{n+1})}  + \frac{\Delta t}{4} \sum_{i=0}^{n}|||(u_h^{i+1} + u_h^i)|||^{2}_{L^2 \left(\Omega_{t^{i+1/2}} \right)} \\
&\leq  \left((1+\Delta t \beta_2^0)\| u_h^{0}\|^2_{L^2(\Omega_{0})}  + \Delta t\sum_{i=0}^{n}\left(\frac{2}{\mu}+  \Delta t \right) \|f^{i+1/2}\|^2_{L^{2}(\Omega_{i+1/2})}\right)\\
 &\exp{\left( \Delta   t \sum_{i=1}^{n+1}\frac{\beta_1^i + \beta_2^i}{1-\Delta t(\beta_1^i + \beta_2^i)} \right)}. 
 \end{aligned}
\end{equation*}
The estimate is stable provided,
\begin{align*}
 \Delta t < \frac{1}{\beta^n_1 +\beta_2^n} = \Big( &\Big|\Big|J_{\mathcal A_{{t_{n}},~t_{n-1/2}}}\Big|\Big|_{L_{\infty}(\Omega_{t_n})} ||\nabla \cdot \mathbf{w}_{h}||_{L_{\infty}(\Omega_{t_{n-1/2}})} \\
&+  \Big|\Big|J_{\mathcal A_{{t_{n}},~t_{n+1/2}}}\Big|\Big|_{L_{\infty}(\Omega_{t_n})}||\nabla \cdot \mathbf{w}_{h}||_{L_{\infty}(\Omega_{t_{n+1/2}})}\Big)^{-1}
\end{align*}

\end{proof}
\begin{remark}
 In this case, for consistent ALE-SUPG, transient term in stabilization is on $\Omega_{n+1/2}$ domain while the standard transient term is in $\Omega_{n+1}$, so the stability estimates for consistent ALE-SUPG needs further investigation.
\end{remark}

\subsection{Discrete ALE-SUPG with backward-difference (BDF-2) time discretization}We now consider the backward difference method of order two for temporal discretization.
For the equation~\eqref{standardeqn}, the backward-difference method gives,
\begin{align*}
\frac{3}{2} u^{n+1} - 2u^{n} + \frac{1}{2} u^{n-1} = \Delta t~ f(u^{n+1}, t^{n+1})
\end{align*}
\begin{lemma}
 Stability estimates for non-conservative ALE-SUPG form applying backward-difference
formula: Let the discrete version of~\eqref{assump1} and the assumption~\eqref{assump2} on $\delta_K$ hold true.  Further, assume that $\delta_K \leq \frac{\Delta t}{4}$ then the solution satisfies 
\begin{equation*}
\begin{aligned}
||u&^{n+1}_h||^2_{L^{2}(\Omega_{t^{n+1}})} + ||2u^{n+1}_h - u^{n}_h||^2_{L^{2}(\Omega_{t^{n}})} +  \Delta t \sum_{i=1}^{n+1} |||u^{i}_h|||^2_{L^{2}(\Omega_{t^{i}})}\\
& \left((1+\Delta t \alpha_2^0)\| u_h^{0}\|^2_{L^2(\Omega_{t^0})} + ||2u_h^1 - u_h^0||^2_{L^2(\Omega_{t^1})} + \Delta t\sum_{i=1}^{n+1}\left(\frac{2}{\mu}+  \frac{\Delta t}{2}\right) \|f^{i}\|^2_{L^{2}(\Omega_{t^i})}\right)\\ 
&~~~\exp{\left( \Delta t \sum_{i=1}^{n+1}\frac{2\alpha_1^i + \alpha_2^i}{1-\Delta t(2\alpha_1^i + \alpha_2^i)} \right)}.
\end{aligned}
\end{equation*}
\end{lemma}
\begin{proof}
 Applying the backward difference temporal discretization to the semi-discrete equation
$(5)$ with the test function $u_h^{n+1},$  we get
\begin{align*}
&\left (\ds\frac{3}{2}u^{n+1}_h -  2u^{n}_h + \frac{1}{2}u^{n-1}_h, u^{n+1}_h \right)_{\Omega_{t ^{n+1}}} + \Delta t~ a^{n+1}_{SUPG}(u^{n+1}_h,u_h^{n+1})\\
&\qquad\qquad\quad - \frac{\Delta t}{2} \int_{\Omega_{t^{n+1}}}  \bw^{n+1}_h \cdot \nabla((u_h^{n+1})^2)~  dx  
   = \Delta t\int_{\Omega_{t^{n+1}}} f^{n+1} u_h^{n+1}~   dx \\
& \qquad\qquad\quad+ \Delta t\sum_{K  \in \mathcal{T}_{t^{n+1}}} \delta_{K}  \int_{K} f^{n+1} ~\left(\mathbf{b-w}^{n+1}_h\right) \cdot\nabla u_h^{n+1} ~ dK.
\end{align*}
The first term can be written as,
\begin{equation}\label{simplfisterm}
\begin{aligned}
 &\left(\frac{3}{2}u^{n+1}_h - 2u^{n}_h + \frac{1}{2}u^{n-1}_h, u^{n+1}_h \right)_{\Omega_{t^{n+1}}} = \frac{1}{4}\Big(||u^{n+1}_h||^2_{L^{2}(\Omega_{t^{n+1}})} - ||u^{n}_h||^2_{L^{2}(\Omega_{t^{n}})} \\ 
&\qquad + ||2u^{n+1}_h - u^{n}_h||^2_{L^{2}(\Omega_{t^{n+1}})}     
- ||2u^{n}_h - u^{n-1}_h||^2_{L^{2}(\Omega_{t^{n}})} + ||u^{n+1}_h - 2u^{n}_h + u^{n-1}_h||^2_{L^{2}(\Omega_{t^{n+1}})}\\
&\qquad + \int_{t^n}^{t^{n+1}} \int_{\Omega_t} \nabla\cdot\bw^{n+1}_h(t) \Big[(u^{n})^2 + \Big(u^{n} - u^{n-1}\Big)^2 \Big]\Big)
\end{aligned}
\end{equation}
substitute equation~\eqref{simplfisterm} for the first term and using the same estimates as we worked in previous sections, the fully discrete equation becomes
\begin{align*}
 \frac{1}{4}&\left(||u^{n+1}_h||^2_{L^{2}(\Omega_{t^{n+1}})} + ||2u^{n+1}_h - u^{n}_h||^2_{L^{2}(\Omega_{t^{n+1}})} +||u^{n+1}_h - 2u^{n}_h + u^{n-1}_h||^2_{L^{2}(\Omega_{t^{n+1}})}\right)\\
& \qquad +\frac{\Delta t}{4}|||u^{n+1}_h|||^2_{L^{2}(\Omega_{t^{n+1}})}
 \leq \frac{1}{4}\left(||u^{n}_h||^2_{L^{2}(\Omega_{t^{n}})} + ||2u^{n}_h - u^{n-1}_h||^2_{L^{2}(\Omega_{t^{n}})}\right)\\
&\qquad + \frac{1}{4} \int_{t^n}^{t^{n+1}} \int_{\Omega_t} \nabla\cdot\bw^{n+1}_h(t) \Big[(u^{n})^2 + (u^{n} - u^{n-1})^2 \Big]+\frac{2\Delta t}{\mu} ||f^{n+1}||^2_{{L_2(\Omega_{t^{n+1}})}}\\
&\qquad  + \frac{\Delta t}{2} ||\nabla\cdot\bw^{n+1}_h||_{L_{\infty}(\Omega_{t_{n+1}})}||u^{n+1}_h||^2_{L^{2}(\Omega_{t^{n+1}})}   +2\Delta t \sum_{K \in \mathcal{T}_{t^{n+1}}} \delta_{K} ||f^{n+1}||^2_{K}.
   \end{align*}
Summing over the index $i = 0, 1, 2, ...n,$ and using the same notations as in implicit Euler case, we get
\begin{align*}
&||u^{n+1}_h||^2_{L^{2}(\Omega_{t^{n+1}})} + ||2u^{n+1}_h - u^{n}_h||^2_{L^{2}(\Omega_{t^{n+1}})} +||u^{n+1}_h - 2u^{n}_h + u^{n-1}_h||^2_{L^{2}(\Omega_{t^{n+1}})}\\
&~~\qquad  +\Delta t \sum_{i=0}^n|||u^{i+1}_h|||^2_{L^{2}(\Omega_{t^{i+1}})} \leq 2 \Delta t \alpha_1^{n+1}||u^{n+1}_h||^2_{L^{2}(\Omega_{t^{n+1}})}+ ||u^{0}_h||^2_{L^{2}(\Omega_{t^{0}})} \\
&~~\qquad +||2u^{1}_h - u^{0}_h||^2_{L^{2}(\Omega_{t^{1}})}+ \Delta t \alpha_2^0||u^{0}_h||^2_{L^{2}(\Omega_{t^{0}})}+\Delta t \sum_{i=1}^n(2 \alpha_1^{i} + \alpha_2^{i})||u^{i}_h||^2_{L^{2}(\Omega_{t^{i}})}\\ 
&~~ \qquad  
 + 4\Delta t\sum_{i=1}^{n}\left(\frac{2}{\mu}+  \frac{\Delta t}{2}\right) \|f^{i+1}\|^2_{L^{2}(\Omega_{t^{i+1}})} \\
& ~ \qquad~ \leq (1+\Delta t \alpha_2^{0})||u^{0}_h||^2_{L^{2}(\Omega_{t^{0}})} +||2u^{1}_h - u^{0}_h||^2_{L^{2}(\Omega_{t^{1}})} +\Delta t \sum_{i=1}^{n+1}(2 \alpha_1^{i} + \alpha_2^{i})||u^{i}_h||^2_{L^{2}(\Omega_{t^{i}})}\\
& ~~\qquad + 4\Delta t\sum_{i=1}^{n+1}\left(\frac{2}{\mu}+  \frac{\Delta t}{2}\right) \|f^{i}\|^2_{L^{2}(\Omega_{t^i})}.
\end{align*}
By using the Grownwall's lemma, we get the stability estimate provided,
\begin{align*}
 \Delta t < \frac{1}{ \sup_{n \in [0,N]} (2 \alpha_1^n + \alpha_2^n)}.
\end{align*}

\end{proof}

\section{Numerical results}
This section presents the numerical results for the proposed SUPG finite element schemes.
We consider a boundary and interior layer problem in a time-dependent domain. The piecewise quadratic finite elements are used for the spatial discretization. The first order backward Euler and second order Crank-Nicolson method are used for the temporal discretization. Numerical solution obtained with the standard Galerkin and the SUPG method are presented. All computations are performed using an unstructured triangular mesh.

\subsection{Example}
In this example, a typical fluid structure interaction problem, that is a flow passing through a rectangular structure (beam), which deforms with time, has been considered. A predefined adaptive mesh with a high resolution near the deforming structure is considered. Nevertheless, the mesh is comparatively coarser away from the structure. Further, the tip of the beam is considered to be semi-circular to do away with the singularities that might occur due to the sharp corners. The mesh movement is handled using the arbitrary Lagrangian Eulerian (ALE) approach.

The flow is being directed by a prescribed sinusoidal movement of the beam, and is given by
\[
d = 0.75(x-0.5)^2 \sin(2\pi t/5) ,\quad \theta = \tan^{-1}\left(\frac{y}{x-0.5}\right).
\]
Here, the time-dependent coordinates $(x_1,x_2)$ are defined as 
\[ x(Y, t) =  \mathcal{A}_t(Y) : \left\{
  \begin{array}{l l}
    x_1 = Y_1 + 0.05(0.25~ d~ \tan \theta - Y_2~ \sin \theta) & \\
    x_2 = Y_2 + 0.05  d.
  \end{array} \right.    
\]
Now, let the time-dependent rectangular structure (beam)
\[
 \Omega_t^S :=\{(-0.5,0.5)\times(-0.5,0.5)\}~ \cup~ \{(0.5,4.5)\times(-0.03,0.03)\}
\]
and the two-dimensional channel that excludes the oscillating beam $\Omega_t^S$ is,
\[
 \Omega_t:=\{(-5,18)\times(-5,5)\}\setminus \bar{\Omega}_t^S.
\]
Further,  we define  $\Gamma_N:=
\{15\}\times(-5,5)$ as the   out flow boundary and $\Gamma_D:= \partial\Omega_t \setminus \Gamma_N$ as the Dirichlet boundary.

\begin{figure}[ht!]
\begin{center}
\unitlength1cm
\begin{picture}(11.5,6.)
\put(-.5,-0.5){\makebox(6,6){\includegraphics[scale=0.24]{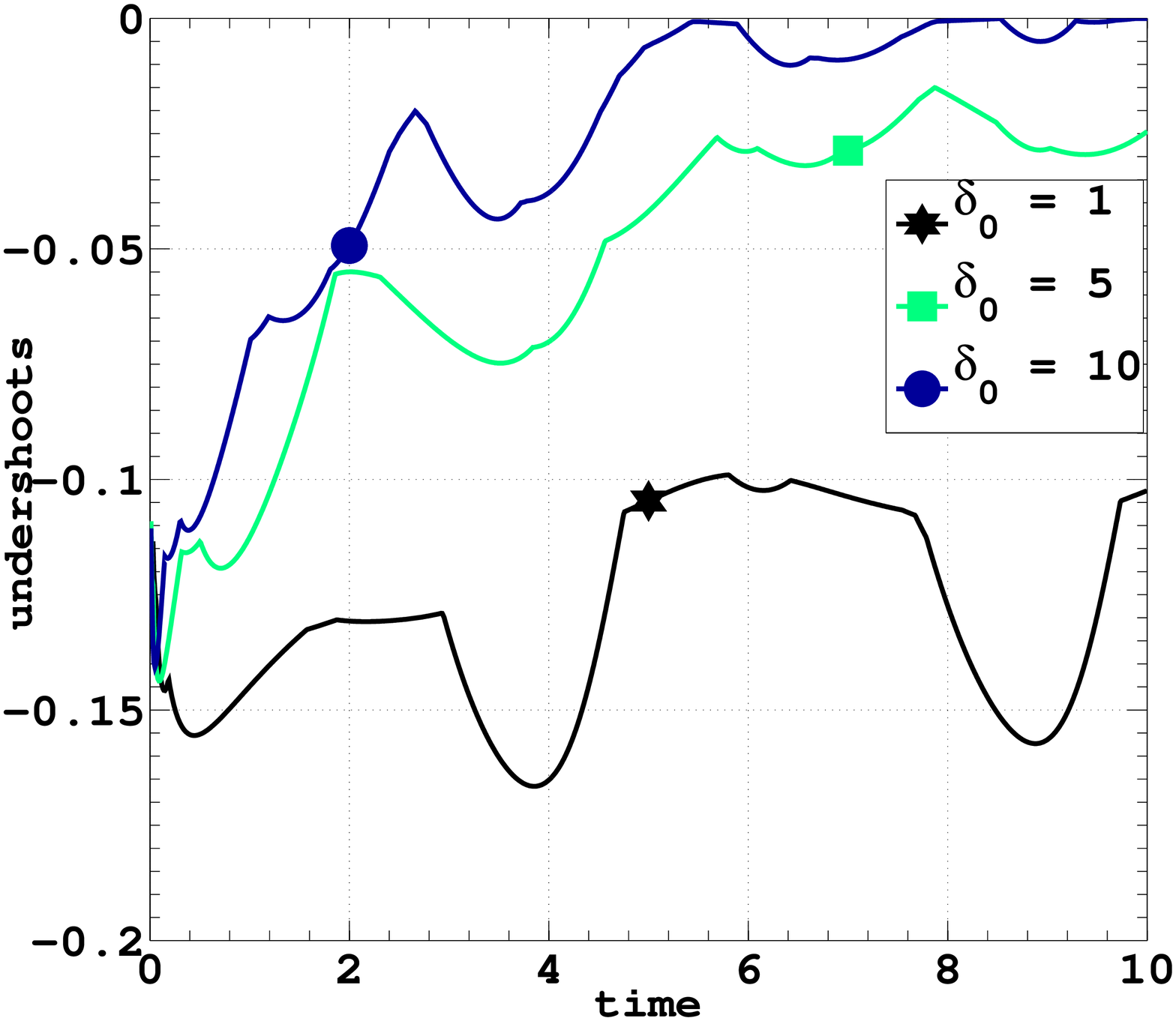}}}
\put(6,-0.5){\makebox(6,6){\includegraphics[scale=0.24]{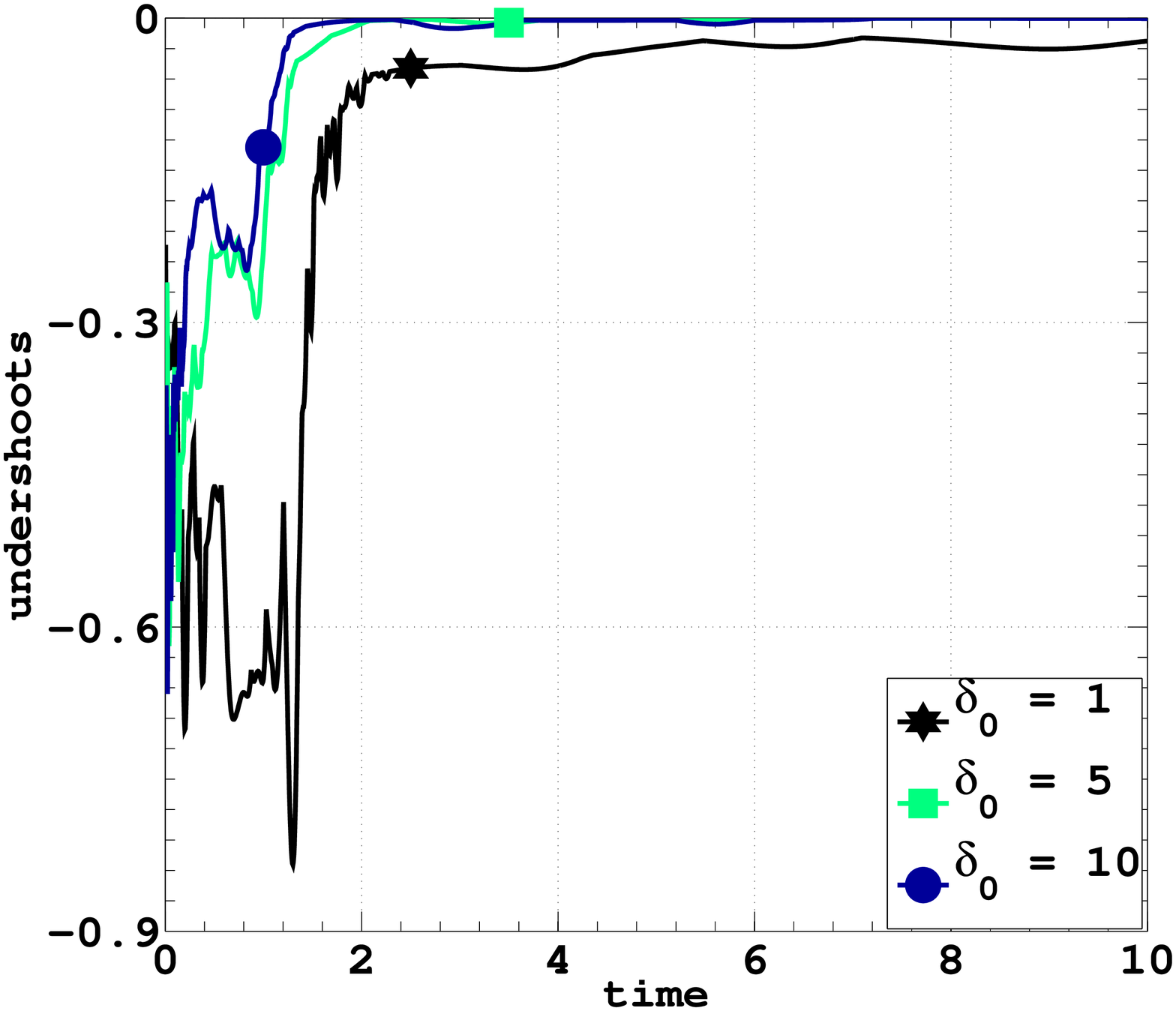}}}
\put(2,5.2){(a)}
\put(9,5.2){(b)}
\end{picture}
\end{center}
\caption{The observed undershoots in the SUPG solution for different values of $\delta_0$. Implicit Euler (a), and Crank-Nicolson (b). \label{Ex2Under}}
\end{figure}
\begin{figure}[ht!]
\begin{center}
\unitlength1cm
\begin{picture}(11.5,6.)
\put(-.5,-0.5){\makebox(6,6){\includegraphics[scale=0.24]{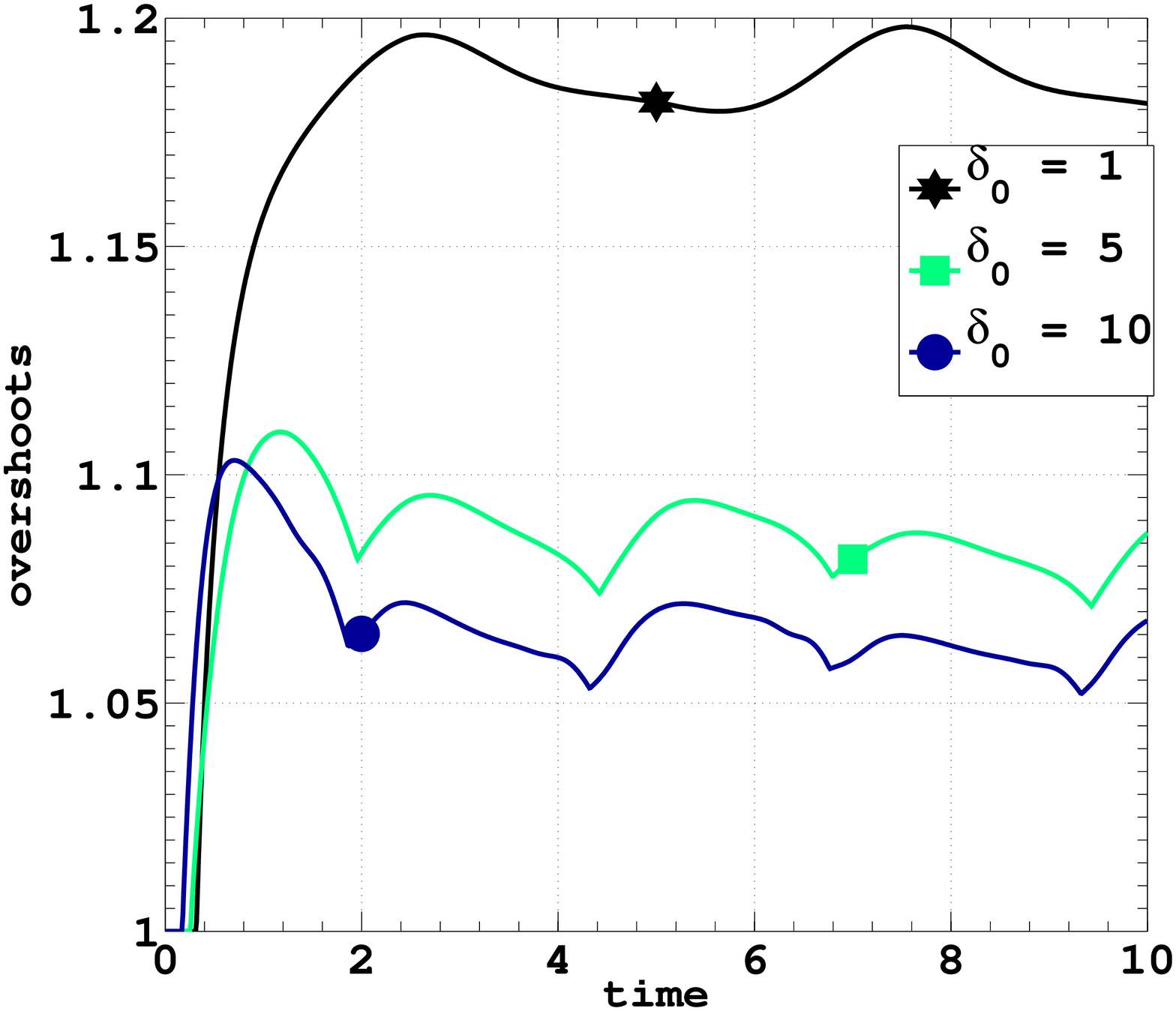}}}
\put(6,-0.5){\makebox(6,6){\includegraphics[scale=0.24]{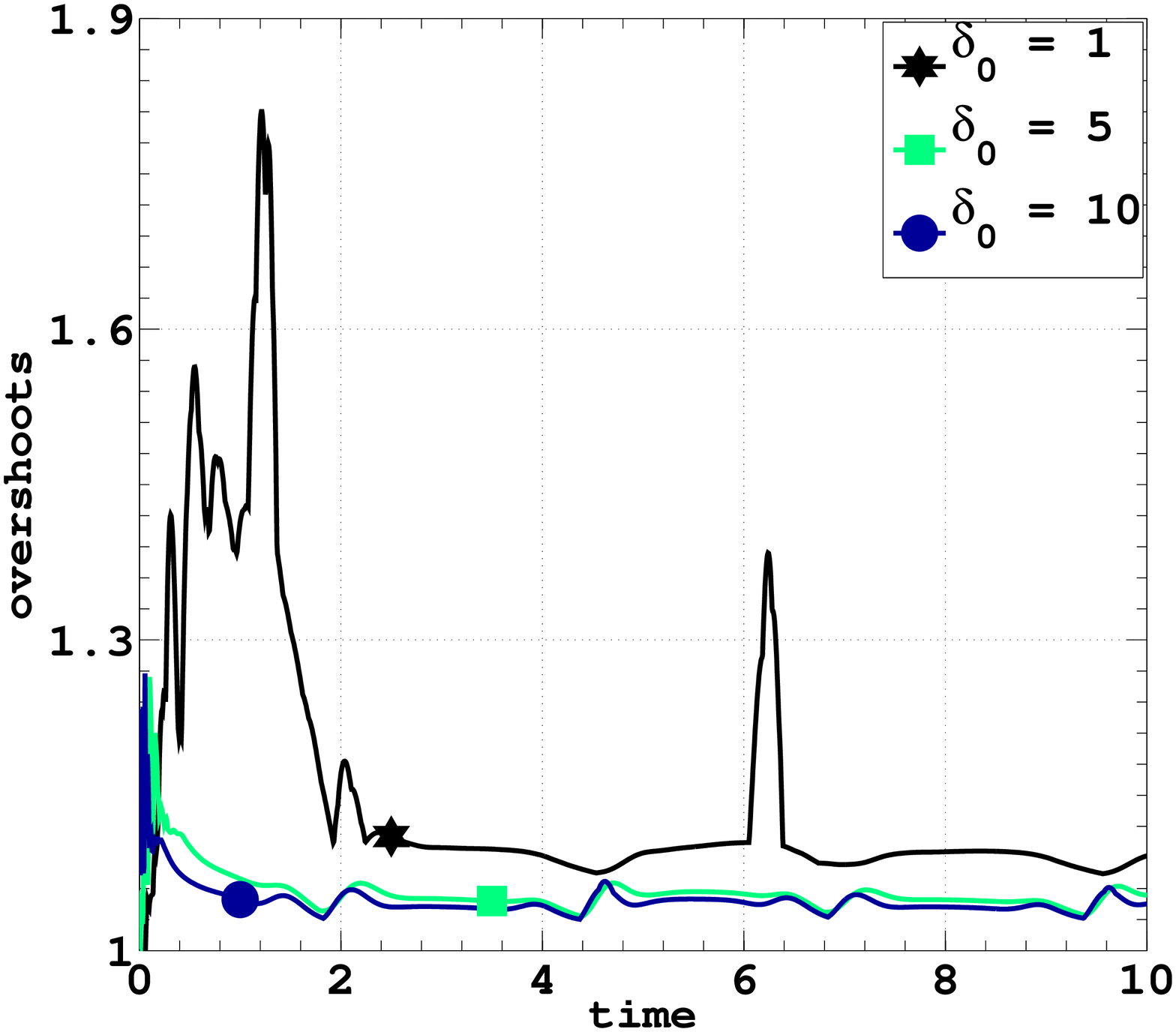}}}
\put(2,5.2){(a)}
\put(9,5.2){(b)}
\end{picture}
\end{center}
\caption{The observed overshoots in the SUPG solution for different values of $\delta_0$. Implicit Euler (a), and Crank-Nicolson (b). \label{Ex2Over}}
\end{figure}

 \begin{figure}[ht!]
\begin{center}
\unitlength1cm
\begin{picture}(11.5,6.)
\put(-.5,-0.5){\makebox(6,6){\includegraphics[scale=0.24]{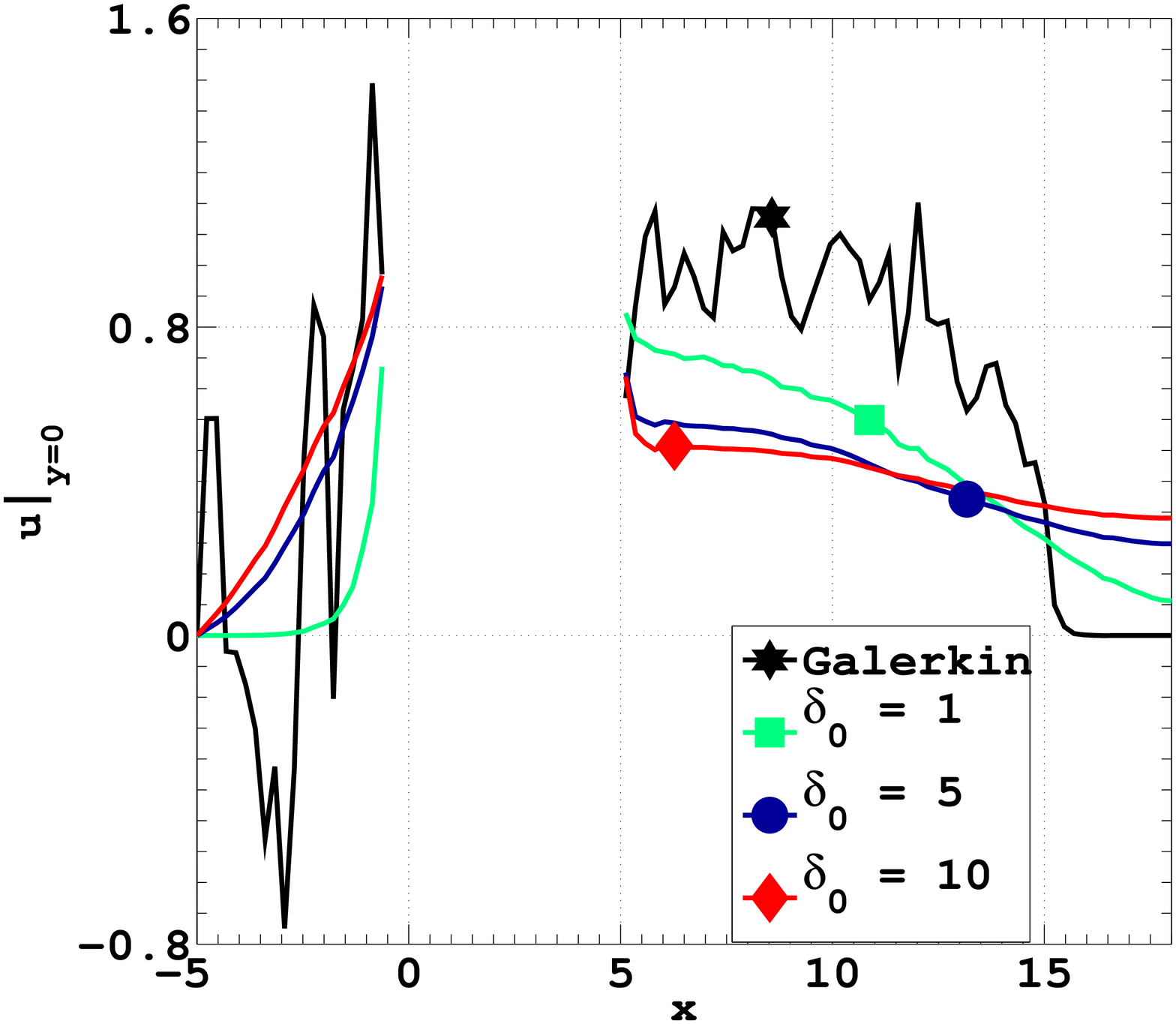}}}
\put(6,-0.5){\makebox(6,6){\includegraphics[scale=0.24]{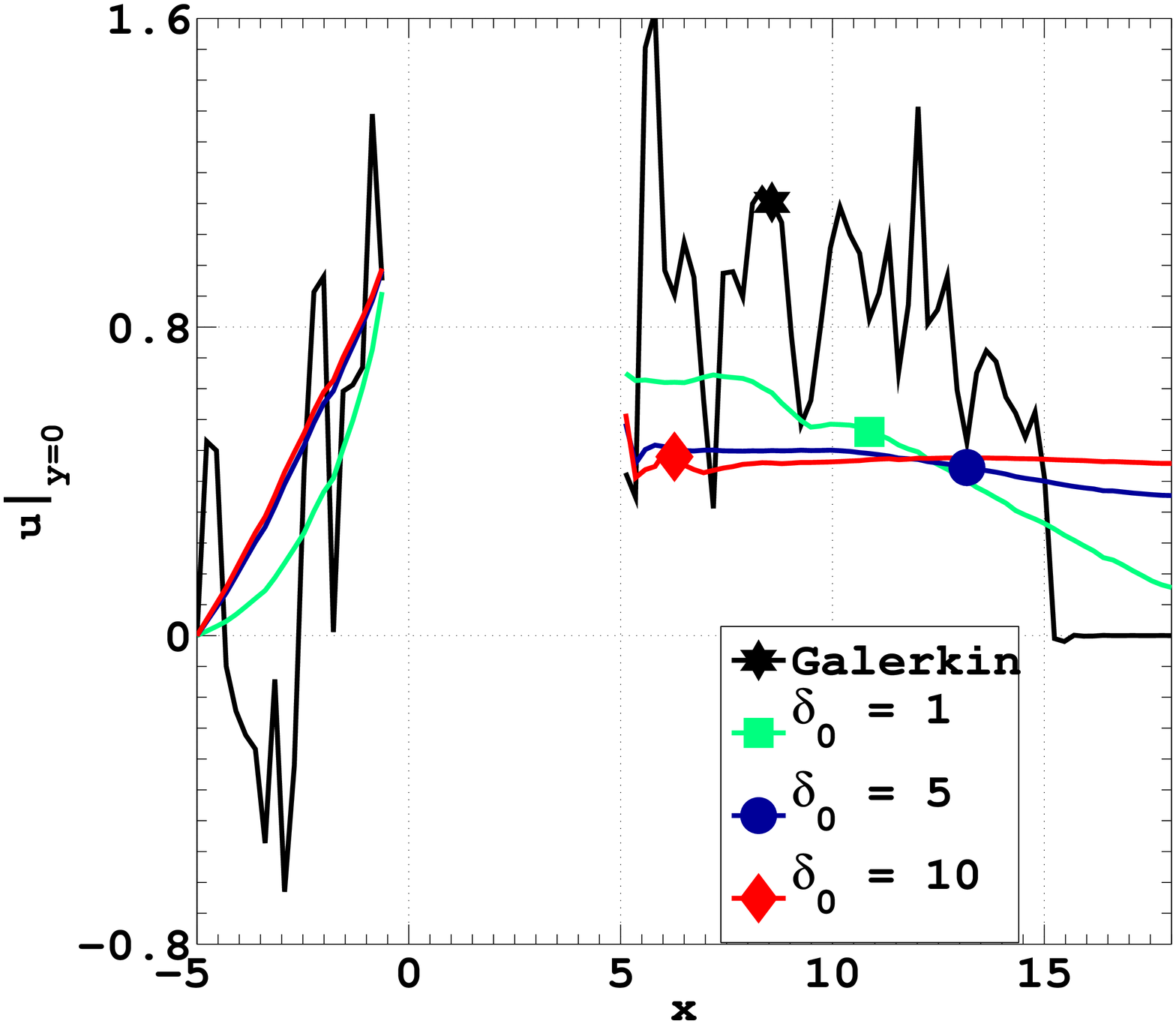}}}
\put(2,5.2){(a)}
\put(9,5.2){(b)}
\end{picture}
\end{center}
\caption{SUPG solution over the line $y=0$ of the Example for different values of $\delta_0$. Implicit Euler (a), and Crank-Nicolson (b). \label{Ex2y0}}
\end{figure}
We solve a scalar (energy) equation 
\[
 \ds\frac{\partial u}{\partial t} - \epsilon\Delta u +  \mathbf b \cdot\nabla u   = 0   \qquad \text{ in} \,\ (0,\rm{T}] \times \Omega_t 
\]
with $\epsilon=10^{-6}$ and  $\mathbf {b}(x_1,x_2)=(1,0)^T$. In this model, we impose the homogeneous Neumann condition on $\Gamma_N$, and      
\[
u_D(x_1,x_2) =
	\begin{cases}
	1 \qquad \text{on}~ \partial \Omega_t^S, \\
	0 \qquad   \text{else}
	\end{cases}
\]
on the Dirichlet boundary. 
Since the solid beam bends (up and down) periodically, the position of the boundary and the interior layers also changes with time.
The  computations are performed until the dimensionless time $\text{T}=10$ with the time step $\Delta t=0.01$. We use the elastic-solid update technique to handle the mesh movement that occurs due to the oscillations of the solid disc see~\cite{shw015}.  At each time step, we first compute the displacement of the disc. We then solve the linear elastic equation in $\Omega_{t^n}$ to compute the inner points' displacement by considering the displacement on $\partial \Omega_{t^{n+1}}^S$ as the Dirichlet value. This elastic update technique avoids the re-meshing during the entire simulation. 

As expected the solution obtained with the standard Galerkin discretization consists spurious oscillations and instabilities, see the Galerkin solution in figure~\ref{Ex2y0}. 
We next perform an array of computations with difference values of  $\delta_0$. The observed undershoots and the overshoots for different values of $\delta_0$  are plotted in Figure~\ref{Ex2Under} and~\ref{Ex2Over}, respectively. It is observed that the undershoots and overshoots are less in the Euler's method (note that the scaling of figures are different). Apart from that increasing value of $\delta_0$ will improve the undershoots/overshoots more in case of implicit Euler than that of Crank-Nicolson scheme. So the optimal choice (based on the undershoots and overshoots), for both the Euler and Crank-Nicolson methods is $\delta_0=5$. The sequence of solutions obtained with the SUPG discretization at different instances $t=0.05, 3.9, 6.2, 10$ are plotted in Figure~\ref{Ex2SupgEuler}. Though, the SUPG approximation suppressed the spurious oscillations in the numerical solution almost, there are very small  overshoots and undershoots (around $10 \%$) for the chosen $\delta_0=5$.  We could reduce these undershoots and overshoots by increasing of $\delta_0$ further, however, it will smear the solution. This is a well known behavior of the SUPG method in stationary domains. Nevertheless, the purpose of the stabilized approximation is achieved, and SUPG discretization can be used for boundary and interior layer problems in time-dependent domains to obtain a stable solution.

 \begin{figure}[ht!]
\begin{center}
\unitlength1cm
\begin{picture}(14,9.5)
\put(0,5.5){\makebox(6,6){\includegraphics[width=6cm]{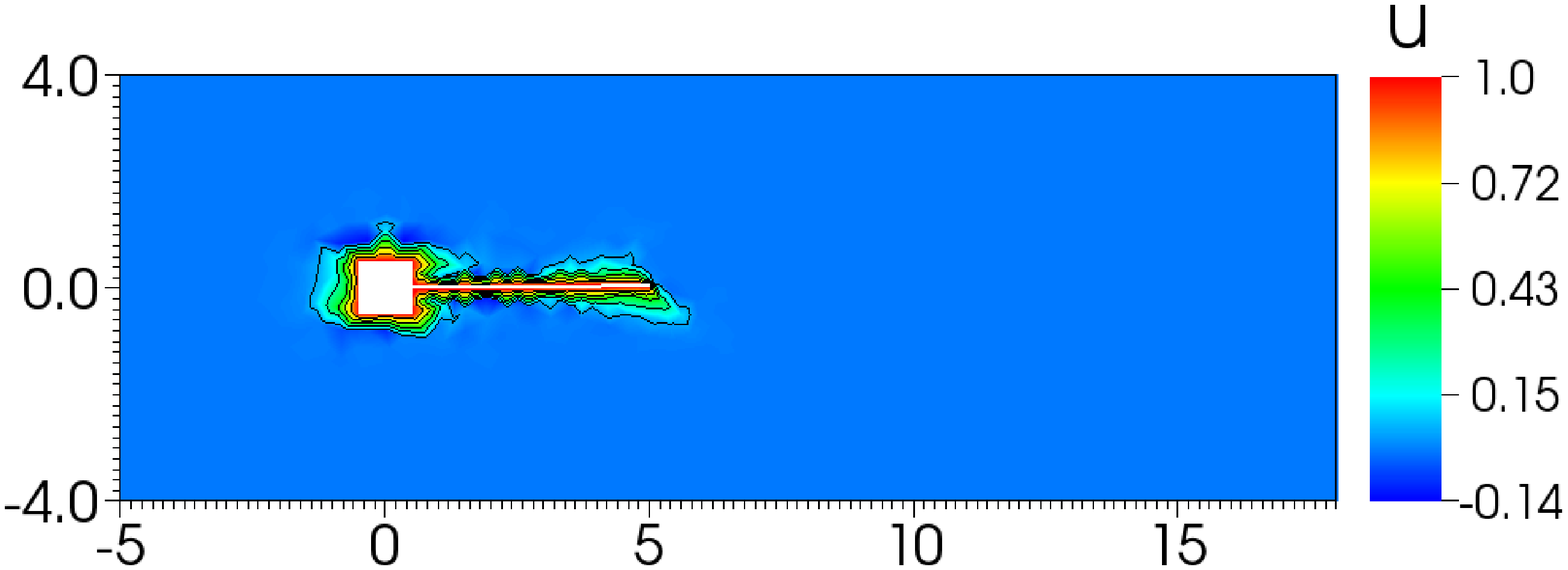}}}
 \put(6.21,5.5){\makebox(6,6){\includegraphics[width=6cm]{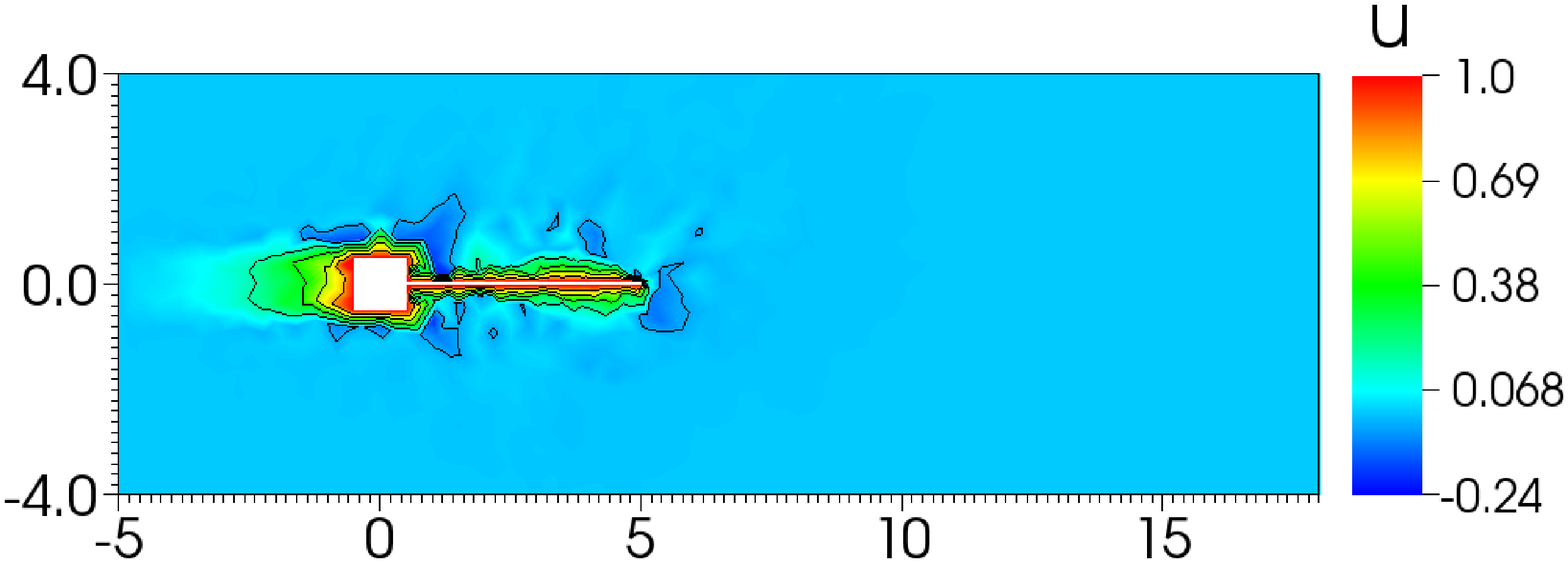}}}
\put(0,3.){\makebox(6,6){\includegraphics[width=6cm]{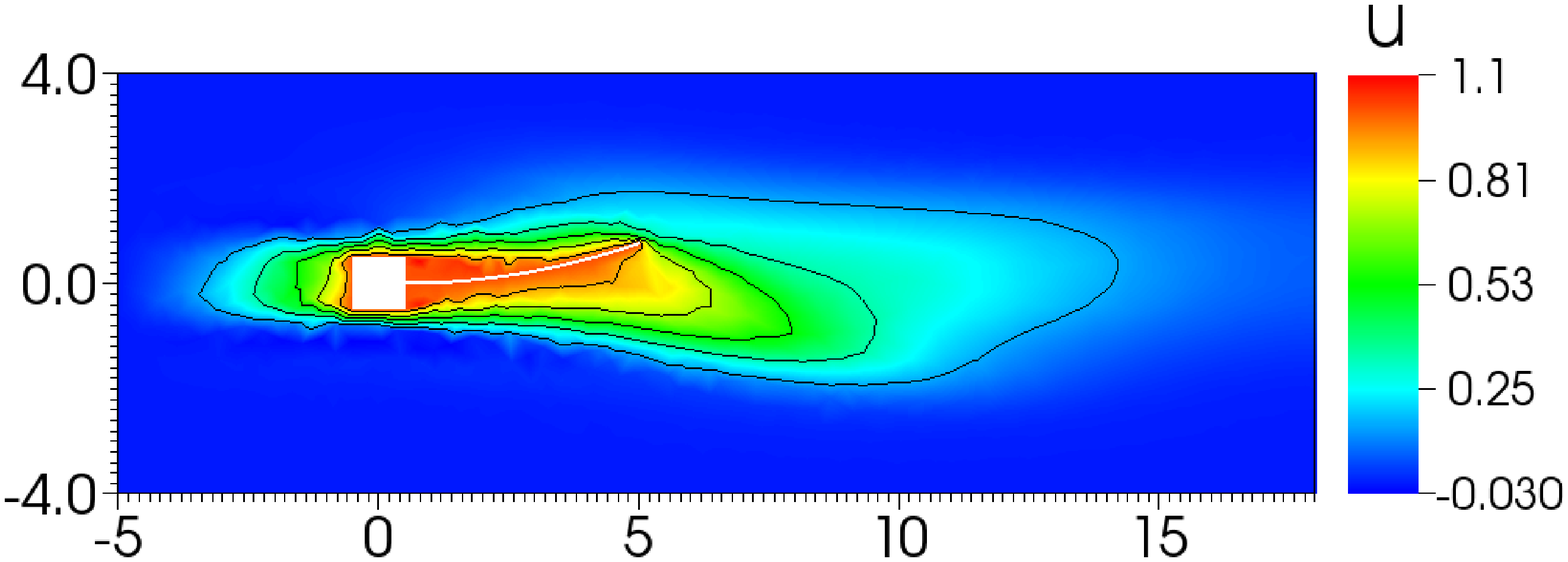}}}
 \put(6.21,3.){\makebox(6,6){\includegraphics[width=6cm]{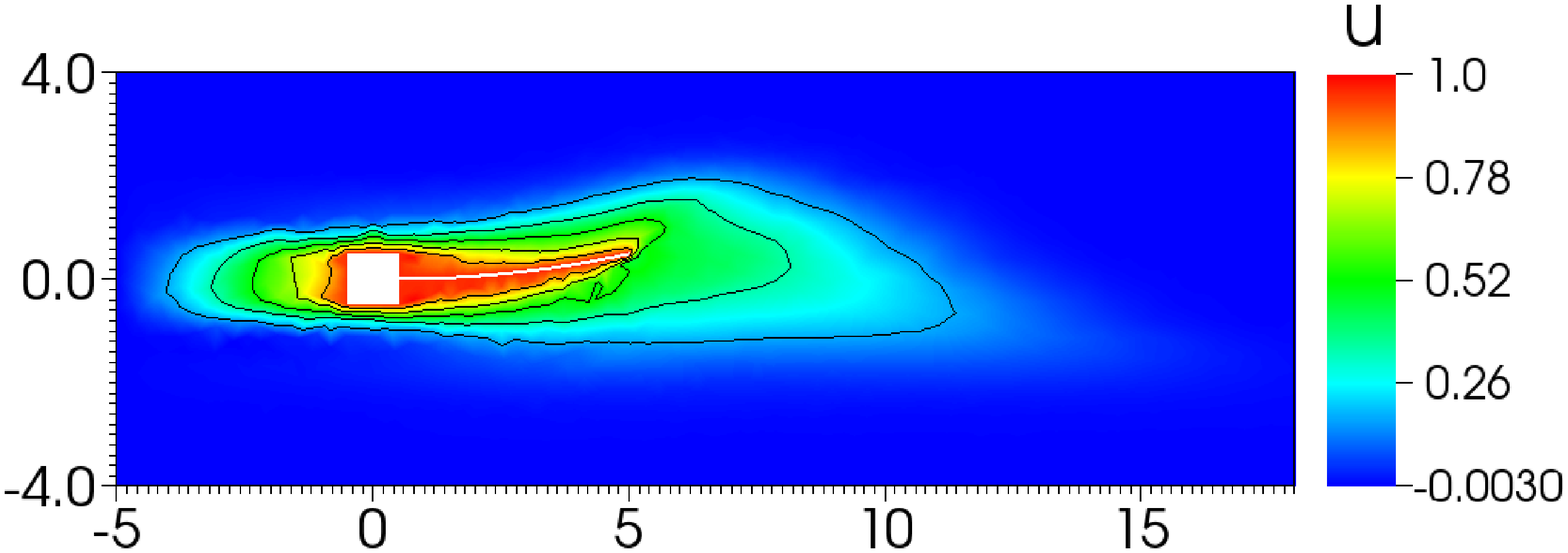}}}
\put(0,.5){\makebox(6,6){\includegraphics[width=6cm]{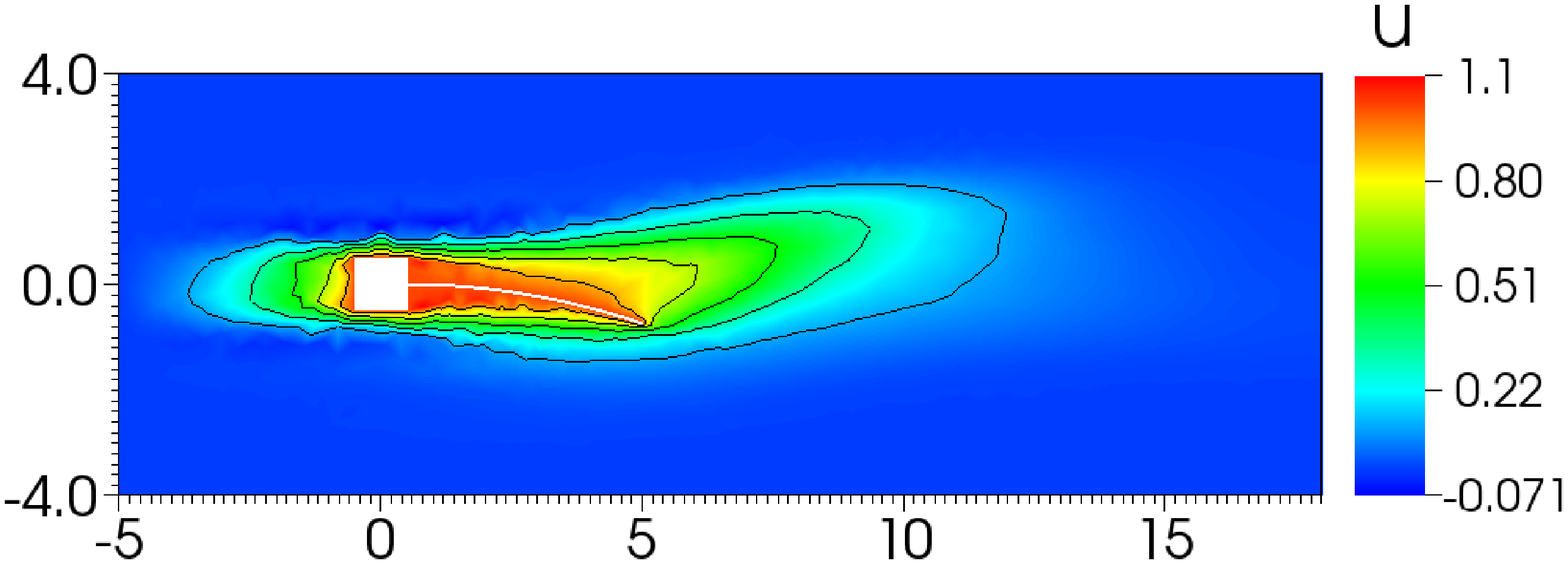}}}
 \put(6.21,.5){\makebox(6,6){\includegraphics[width=6cm]{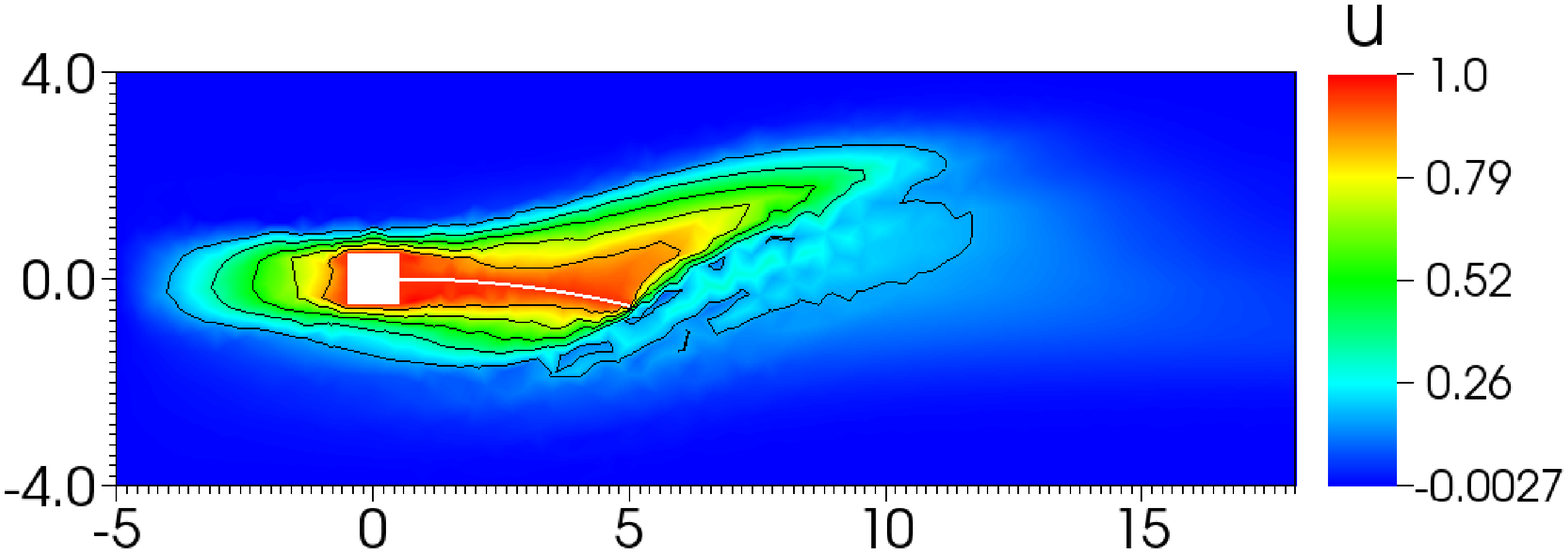}}}
\put(0,-2.){\makebox(6,6){\includegraphics[width=6cm]{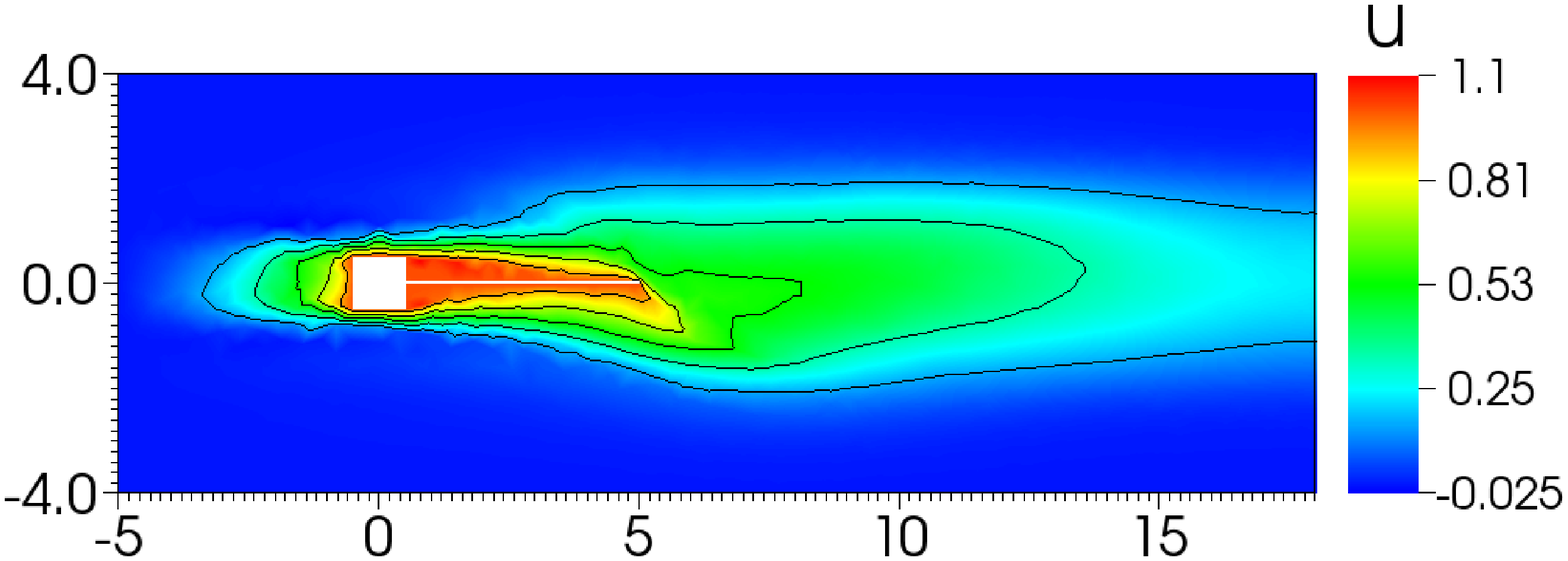}}}
\put(6.21,-2.){\makebox(6,6){\includegraphics[width=6cm]{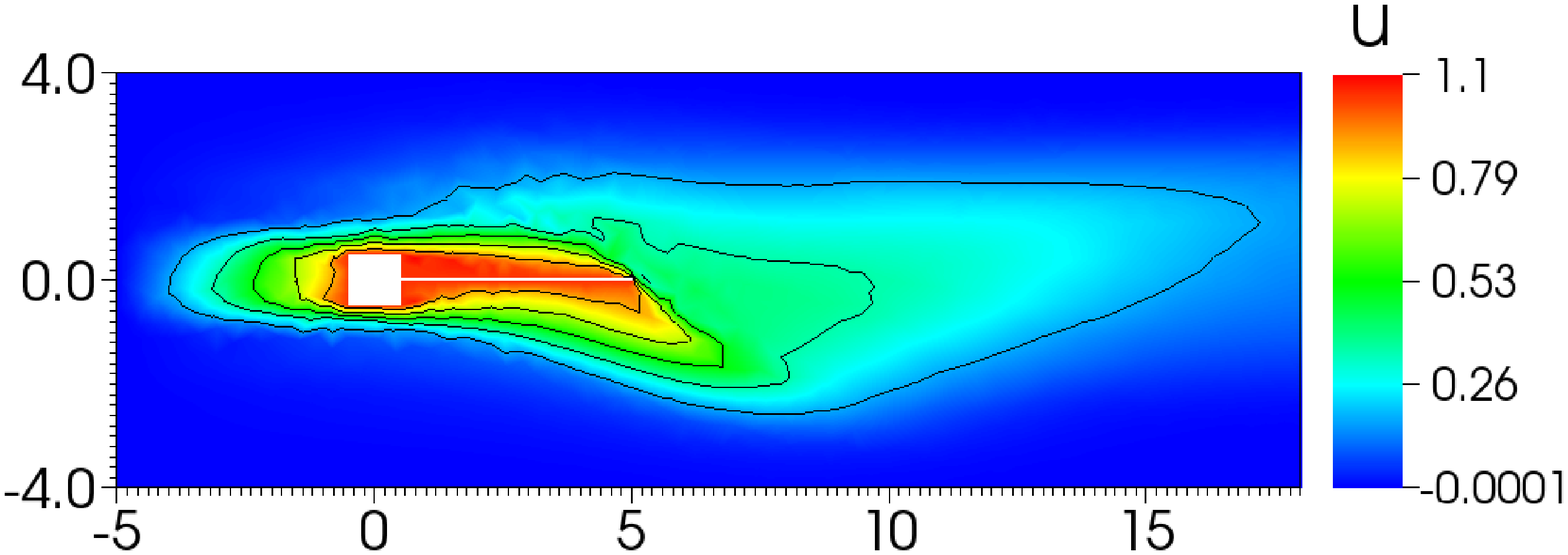}}}
\end{picture}
\end{center}
\caption{The sequence of solutions obtained with the SUPG $\delta_0 = 10$ method for Implicit Euler case(left) and Crank-Nicolson case(right) at different instance $t = 0.05, 3.9, 6.2, 10 $. \label{Ex2SupgEuler}}
\end{figure}
\section{Summary} ALE-SUPG finite element scheme for convection dominated transient convection-diffusion equation in time-dependent domains is presented in this paper. The deformation of the domain is handled using the arbitrary Lagrangian Eulerian (ALE) approach. The stability estimates of the  backward Euler, Crank-Nicolson and backward difference(BDF-2) temporal discretizations with non-conservative ALE-SUPG finite element method are derived. The scheme is validated with flow over an oscillating beam with convection dominant case. Further, the influence of the SUPG stabilization parameter on the solution with ALE-SUPG finite element scheme is presented. It is observed that the Crank-Nicolson scheme is less dissipative than the Implicit Euler method. Further, the influence of the SUPG stabilization method differs in the backward Euler and Cranck-Nicolson methods. The undershoots/overshoots in the backward Euler scheme is more sensitive to the stabilization parameter $\delta_k$ than in the Crank-Nicolson scheme.

 \bibliographystyle{spmpsci} 
 \bibliography{masterlit}

%
%
%
\end{document}